\documentclass[journal]{IEEEtran}
\ifCLASSINFOpdf
\else
\fi

\usepackage{amsmath,amssymb,bbm,amsthm}
\usepackage{color}
\usepackage{mathtools}

\usepackage{microtype}
\usepackage{graphicx}
\usepackage{subfigure}
\usepackage{booktabs} %
\usepackage{multicol, multirow}
\usepackage{float}

\usepackage{caption}
\usepackage{hyperref}

\usepackage{dsfont}

\usepackage{bbold}

\usepackage{dsfont}

\usepackage{bbold}

\DeclareMathOperator*{\R}{\mathbb{R}}
\DeclareMathOperator*{\E}{\mathbb{E}}
\DeclareMathOperator*{\Rn}{\mathbb{R}^n}
\DeclareMathOperator*{\Rd}{\mathbb{R}^d}

\newcommand{\RR}{\mathbb{R}}

\DeclareMathOperator*{\argmin}{arg\,min}

\DeclareMathOperator{\prox}{Prox}

\newcommand{\vareps}{\varepsilon}
\DeclarePairedDelimiter{\dotp}{\langle}{\rangle}
\newtheorem{lemma}{Lemma}
\newtheorem{theorem}{Theorem}
\newtheorem{remark}{Remark}
\newtheorem{corollary}{Corollary}
\newtheorem{definition}{Definition}
\newtheorem{assumption}{Assumption}
\newtheorem{proposition}{Proposition}

\newtheorem{example}{Example}

\makeatletter
\newcounter{procedure}

\makeatother

\usepackage{microtype}
\usepackage{graphicx}
\usepackage{subfigure}
\usepackage{booktabs} %

\usepackage{hyperref}

\usepackage{amsmath,amssymb,bbm,amsthm}
\usepackage{color}
\usepackage{mathtools}

\usepackage{dsfont}

\usepackage{bbold}

\usepackage{algorithm, algorithmic}

\begin{document}
%
\title{Decentralized  Learning  with  Lazy  \\and  Approximate Dual Gradients}
%
%
%

\author{Yanli Liu,\and
        Yuejiao Sun,\and
        and~Wotao Yin
        \thanks{Y. Liu, Y. Sun, and W. Yin are with University of California, Los Angeles. This work was supported in part by AFOSR MURI grant FA9550-18-10502
and ONR grant N0001417121.}
        }

\maketitle

\begin{abstract}
This paper develops algorithms for decentralized machine learning over a network, where data are distributed, computation is localized, and communication is restricted between neighbors. A line of recent research in this area focuses on improving both computation and communication complexities. The methods SSDA and MSDA \cite{scaman2017optimal} have optimal communication complexity when the objective is smooth and strongly convex, and are simple to derive. However, they require solving a subproblem at each step. We propose new algorithms that save computation through using (stochastic) gradients and saves communications when previous information is sufficiently useful. Our methods remain relatively simple --- rather than solving a subproblem, they run Katyusha for a small, fixed number of steps from the latest point. An easy-to-compute, local rule is used to decide if a worker can skip a round of communication. Furthermore, our methods provably reduce communication and computation complexities of SSDA and MSDA. In numerical experiments, our algorithms achieve significant computation and communication reduction compared with the state-of-the-art.
\end{abstract}



%
\IEEEpeerreviewmaketitle

\section{Introduction}
\label{sec: introduction}

Consider $n$ workers in a connected network $\mathcal{G}$, where each worker $i$ has a local minimization objective
$f_i(\theta)=\frac{1}{m}\sum_{j=1}^m f_{i,j}(\theta)$. Assume all workers have the same $m$ for simplicity. We aim to solve the distributed learning problem 
\begin{align}
\label{equ: original primal problem}
\mathop{\mathrm{minimize}}_{\theta\in \Rd} f(\theta)\coloneqq \sum_{i=1}^n f_i(\theta)= \sum_{i=1}^n \frac{1}{m}\sum_{j=1}^m f_{i,j}(\theta)
\end{align}
using a decentralized method. By ``decentralized'', we mean the global objective \eqref{equ: original primal problem} is achieved through local computation and between-neighbor communication. 

In decentralized computation, there is no central server to collect or distribute information, so it avoids a communication hot-spot and the potential disastrous failure of the central server. With a well-connected network, decentralized computation is robust since a small number of failed workers or communication links do not disconnect the network. Being decentralized also makes it harder for a malicious worker or eavesdropper to collect information, so it is relatively  secure. Because of these attractions, decentralized computation has been widely adopted in sensor networks, multi-agent controls, distributed machine learning, and recently federated learning.

Developing a decentralized method for \eqref{equ: original primal problem} can be reduced to solving a constrained optimization problem \cite{shi2015extra}. Let $\theta_i$ denote worker $i$'s \emph{local copy} of $\theta$. A decentralized method ensures each $\theta_i$ to equal its neighbors' copies and, consequently, equal all other copies in the network. Write $\Theta=(\theta_1,\theta_2,\ldots, \theta_n)\in\RR^{d\times n}$. Call $\theta_1=\dots= \theta_n$ \emph{consensus}. With a proper symmetric, positive semi-definite matrix $U$ (which is defined later) and $\sqrt{U}$ (satisfying $\sqrt{U}\sqrt{U}=U$), we can express consensus equivalently as $\Theta \sqrt{U}=0$. Let $F(\Theta)=\sum\limits_{i=1}^n f_i(\theta_i)$. We can equivalently rewrite problem \eqref{equ: original primal problem} as the constrained problem
\begin{align}
\label{equ: primal problem}
    \mathop{\mathrm{minimize}}_{\Theta}F(\Theta)\quad\mathrm{subject~to}~~{\Theta \sqrt{U}=0}.
\end{align}
A decentralized method defined for \eqref{equ: primal problem} performs local computation with $f_i$ and between-neighbor communication, which is expressed with the multiplication $\Theta U$. We review these methods in next subsection below.

The cost of communication of a decentralized method corresponds to the number of multiplications by $U$.  If a method can solve all instances in a class of problem \eqref{equ: primal problem} up to an accuracy with fewest $U$-multiplications, we say it is \emph{communication optimal}. It is challenging to find such a method for \eqref{equ: primal problem} since $\sqrt{U}$ appears in the constraints. The first communication-optimal method is SSDA and its variant MSDA \cite{scaman2017optimal}. Their optimality is established for smooth and strongly-convex $F$ and requires a subproblem oracle. Specifically, SSDA and MSDA apply Nesterov's accelerated gradient method \cite{nesterov2013introductory} to the dual of \eqref{equ: primal problem}. This is a simple yet effective idea. But, since computing the dual gradient is equivalent to  solving certain subproblems involving minimizing $f_i$ at each worker $i$, the computational complexities of SSDA and MSDA are given in the number of subproblems solved, even if $\nabla F$ is available. When subproblems are solved approximately in practice, it is unclear whether SSDA and MSDA remain communication-optimal.

In this paper, we develop decentralized methods that maintain communication optimality but use gradients of $F$ instead of solving subproblems. In addition, when the objective has a finite-sum structure, i.e., when $m>1$, our methods can use stochastic gradients of $F$. Specifically, like SSDA and MSDA, our methods solve \eqref{equ: primal problem} by solving its dual, but unlike them, each iteration of our methods involves a small, fixed number of Katyusha steps, starting from the latest point. Katyusha \cite{allen2017katyusha} is an  accelerated stochastic variance-reduction gradient method for solving a standalone finite-sum problem. While our methods maintain optimal communication complexity, they also achieve the best sample complexity among those based on gradients. 

Being communication optimal means we cannot find ways to save (significantly) more communication on the worst instance of a problem class. But there is often room to improve on general instances. Specifically, the workers can skip sending slowly varying information to their neighbors without slowing down convergence. Motivated by the method \textit{lazily aggregated gradients} or LAG \cite{chen2018lag}, which uses a rule to select communication in a server-worker setting, we follow its idea and develop a new rule for our methods. We analyze our method under the proposed rule and establishes the same worst-case computation and communication complexities. When workers have heterogeneous data, however, the rule leads to a much-reduced communication complexity.

\subsection{Prior art}
\label{subsec: related work}

\subsubsection{Related work}

To save computation, it is advantageous to allow concurrent information exchange. That is, every worker can communicate with all of its neighbors in each communication round. Several popular methods fall into this category, including distributed ADMM (D-ADMM) \cite{mota2013d,shi2014linear}, EXTRA \cite{shi2015extra}, exact diffusion \cite{yuan2018exact}], DIGing \cite{nedic2017achieving}, COLA \cite{he2018cola}, and their extensions. They all enjoy linear convergence with a constant step size. However, their worst-case iteration complexities are not optimal. 

Recently, this issue is partially resolved by the SSDA (Single Step Dual Accelerated method) and MSDA (Multi-Step Dual Accelerated method) proposed in \cite{scaman2017optimal}. In this work, Nesterov's accelerated gradient descent is applied to the dual problem, but assumes that the dual gradients $\nabla f^*_i$ are provided by an oracle. However, for most applications, a lot of computation is required to obtain an accurate dual gradient. To resolve this, \cite{uribe2018dual} proposes to compute $\vareps^2-$approximate dual gradients where $\vareps$  is the final target accuracy. This leads to an overall primal gradient complexity of $\mathcal{O}(\log^2(\frac{1}{\vareps}))$. Recently, the multi-step idea is also applied on nonconvex decentralized problems, and a near optimal communication complexity of $\mathcal{O}(\frac{1}{\varepsilon})$ is obtained \cite{sun2019distributed}. In
\cite{hendrikx2019accelerated}, the authors propose to calculate the proximal mapping of $f^*_{i,j}$ instead, which can be potentially easier. However, these proximal mappings are still not in closed form in general, and it is unclear how accurate they should be in order for the algorithm to converge. Furthermore, the network is required to be symmetric enough (e.g., complete graph or 2-D grid).

In order to exploit the finite sum structure of the data at each worker, several stochastic decentralized algorithms have also been developed to save computation. \cite{koloskova2019decentralized} proposes a decentralized stochastic algorithm with compressed communication but without linear convergence.  DSA \cite{mokhtari2016dsa} and DSBA \cite{shen2018towards} achieve linear convergence, but not at an Nesterov-accelerated rate. Among them, DSBA achieves a better rate but requires an oracle for the proximal mapping of $f_{i,j}$. Under nonconvex settings, \cite{sun2019improving} combines gradient estimation and gradient tracking, and achieves a communication and computation complexity of $\mathcal{O}(\frac{1}{\varepsilon})$.

\subsubsection{Other efforts to save communication}

Recently, many methods have been developed to save communication, which can be categorized as follows: (i) Gradient quantization and sparsification, and (ii) Skipping unnecessary communication rounds.

Gradient quantization technique applies a smaller bit width to lower the gradient’s floating-point precision. It was first proposed in 1-bit SGD \cite{seide20141,strom2015scalable}. Later, QSGD \cite{alistarh2017qsgd} introduced stochastic rounding to ensure the unbiasedness of the estimator. More recently, signSGD with majority vote \cite{bernstein2018signsgd} is developed for the centralized setting. In gradient sparsification, only the information preserving gradient coordinates are communicated (e.g., the ones that are large in magnitude). This idea is first introduced in \cite{seide20141}. Later, the skipped small gradient coordinates are accumulated and communicated when large enough \cite{stich2018sparsified,alistarh2018convergence}. More recently, \cite{wangni2018gradient}, which achieves a balance between the gradient variance and sparsity.

To save communication complexity, a line of work focuses on skipping communication by a fixed schedule. In local SGD methods \cite{lin2018don,stich2019local,yu2018parallel}, communication complexity is reduced by periodic averaging, recently, this strategy is also generalized to the decentralized setting \cite{wang2018cooperative}. However, they all require the data to be i.i.d. distributed across the workers, which is unrealistic for federated learning settings where data distribution is often heterogeneous.

Recently, the dynamic communication-saving strategy called LAG is proposed for the centralized setting \cite{chen2018lag}, which exploits the data heterogeneity rather than suffering from it. As mentioned before, this strategy results in a  provable communication reduction when the data distributions vary a lot across the workers. In this work, we propose an algorithm that generalizes this idea to the decentralized setting. This task is non-trivial since unlike LAG, stale information is no longer applied at the server but all over the network.

\subsection{Our contributions}
\label{subsec: our contributions}

In this work, we propose DLAG and MDLAG, which are stochastic decentralized algorithms that achieve both computation and communication reduction over those of SSDA and MSDA in \cite{scaman2017optimal}, respectively. On the one hand, our methods save computation by using highly inexact dual gradients that are obtained by efficient stochastic methods. Somewhat surprisingly, we show they maintain the convergence rates of SSDA and MSDA. On the other hand, they save communication by generalizing the idea of lazily aggregated gradients \cite{chen2018lag} to the decentralized setting, where each worker communicates with its neighbors only if the old approximate dual gradient in cache is too outdated. Otherwise, the old approximate dual gradient can still be applied for the current update and won't degrade the convergence rate.

In summary, DLAG and MDLAG enjoy the following nice properties (see also Table \ref{table: comparison}). %
\begin{enumerate}
    \item DLAG and MDLAG compute approximate dual gradients efficiently by warm start and cheap subroutines. Convergence is established, and the computation complexity does not depend on the (potentially high) cost of the oracles of exact dual gradient $\nabla f^*_i$ or exact proximal mapping $\prox_{f_{i,j}}$.
    
    \item In addition, DLAG also provably reduces communication complexity compared with the state-of-the-art, thanks to the idea of lazily aggregated gradients.
    \item All these claims are verified numerically.
\end{enumerate}

\begin{center}
\begin{table*}[ht]
\centering
\begin{tabular}{c|l|l|c|c}
\hline
\multirow{2}{*}{Method}          & Use  & Use  & \multirow{2}{*}{Computation complexity}                                                   & \multirow{2}{*}{Communication complexity}                                                              \\
&\multicolumn{1}{c|}{$\nabla f^*_{i}$} & \multicolumn{1}{c|}{$\prox_{f_{i,j}}$} & & \\ \hline\hline
SSDA    &      \multicolumn{1}{c|}{Yes}                         &        \multicolumn{1}{c|}{No}                         & $\tilde{\mathcal{O}}(n p_1\sqrt{\frac{\kappa_F}{\zeta(U)}})$            & $\tilde{\mathcal{O}}(|\mathcal{E}|\sqrt{\frac{\kappa_F}{\zeta(U)}})$             \\ \hline
MSDA    &    \multicolumn{1}{c|}{Yes}                           &      \multicolumn{1}{c|}{No}                           & $\tilde{\mathcal{O}}(n p_1\sqrt{\kappa_{F}})$                           & $\tilde{\mathcal{O}}(|\mathcal{E}|\sqrt{\frac{\kappa_F}{\zeta(U)}})$             \\ \hline
Distributed FGM      &       \multicolumn{1}{c|}{No}                        &   \multicolumn{1}{c|}{No}                              & $\tilde{\mathcal{O}}(n m\kappa_F\sqrt{\frac{1}{\zeta(U)}})$        & $\tilde{\mathcal{O}}(|\mathcal{E}|\sqrt{\frac{\kappa_F}{\zeta(U)}})$             \\ \hline
DSBA &            \multicolumn{1}{c|}{No}                   &            \multicolumn{1}{c|}{Yes}                     & $\tilde{\mathcal{O}}(n p_2(\kappa_{F}+\frac{1}{\zeta(U)}+m))$ & $\tilde{\mathcal{O}}(|\mathcal{E}|(\kappa_{F}+\frac{1}{\zeta(U)}+m))$ \\ \hline
ADFS  &          \multicolumn{1}{c|}{No}                     &             \multicolumn{1}{c|}{Yes}                    &  {\small $\tilde{\mathcal{O}}(n p_2(\sqrt{\frac{\kappa_F}{\zeta(U)}}+\frac{\frac{1}{n}\sum_{i=1}^{n}(m+\sqrt{m\kappa_i})}{\sqrt{\frac{\kappa_{\mathrm{min}}}{\kappa_{\mathrm{max}}}}}))$}                                                                                  & {\small $\tilde{\mathcal{O}}(|\mathcal{E}|(\sqrt{\frac{\kappa_F}{\zeta(U)}}+\frac{\frac{1}{n}\sum_{i=1}^{n}(m+\sqrt{m\kappa_i})}{\sqrt{\frac{\kappa_{\mathrm{min}}}{\kappa_{\mathrm{max}}}}}))$}                                                                                        \\ \hline DLAG (this paper)                &       \multicolumn{1}{c|}{No}                        &          \multicolumn{1}{c|}{No}                       &$\tilde{\mathcal{O}}(n(m+\sqrt{m\kappa_{\mathrm{max}}})\sqrt{\frac{\kappa_F}{\zeta(U)}})$            & $\tilde{\mathcal{O}}(q|\mathcal{E}|\sqrt{\frac{\kappa_F}{\zeta(U)}})$                                                                                       \\ \hline
MDLAG (this paper)                &       \multicolumn{1}{c|}{No}                        &          \multicolumn{1}{c|}{No}                       &$\tilde{\mathcal{O}}(n(m+\sqrt{m\kappa_{\mathrm{max}}})\sqrt{\kappa_F})$            & $\tilde{\mathcal{O}}(|\mathcal{E}|\sqrt{\frac{\kappa_F}{\zeta(U)}})$                                                                                       \\ \hline
\end{tabular}
\caption{Comparison of SSDA \cite{scaman2017optimal}, MSDA \cite{scaman2017optimal}, Distributed FGM \cite{uribe2018dual}, DSBA \cite{shen2018towards}, and ADFS \cite{hendrikx2019accelerated} with our DLAG and MDLAG. We have omitted a $\log(\frac{1}{\vareps})$ factor in $\tilde{\mathcal{O}}$. $\kappa_F$, $\kappa_{\mathrm{min}}$, $\kappa_{\mathrm{max}}$, and $\kappa_i$ are defined in Assumption \ref{assump: smoothness and strong convexity}. $\zeta(U)$ is the normalized eigengap of the network graph defined in Assumption \ref{assump: assumption on network topology}, and $|\mathcal{E}|$ is its number of edges. For SSDA and MSDA, $p_1$ is the complexity of computing an exact $\nabla f^*_i$. For DSBA and ADFS, $p_2$ is the complexity of computing an exact $\prox_{f_{i,j}}$. Depending on the problem, $p_1, p_2$ can be mild and can also be very large. In our DLAG, $q\leq 1$ depends on the distribution of $\mu_i$ across the workers, and it is defined in \eqref{equ: compare communication complexity}.}
\label{table: comparison}
\end{table*}
\end{center}

\section{Notation and assumptions}
\label{sec: preliminaries and assumptions}

Throughout this paper, we use $\|\cdot\|$ for  $\ell_2-$norm of vectors and Frobenius norm of matrices, $\langle\cdot, \cdot\rangle$ stands for dot product. For a symmetric, positive semidefinite matrix $M\in\R^{n\times n}$, 
we define $\sqrt{M}$ by $\sqrt{M}\coloneqq S^T A^{\frac{1}{2}}S$, where $M=S^T A S$ is the eigen-decomposition of $M$. We denote the null space of $M$ by $\mathrm{null}(M)$. $\mathbf{1}$ stands for the all-one vector $(1,1,...,1)^T\in\Rn$.

For $\varphi: \Rn\rightarrow\R$, its conjugate $\varphi^*: \Rn\rightarrow\R$ is defined as:
\[
\varphi^*(y)=\sup_{x\in\Rn}\{\langle y, x\rangle-\varphi(x)\}.
\]

\begin{definition}
\label{def: smoothness}
We say that $\varphi: \Rd \rightarrow \R$ is $L-$smooth with $L\geq 0$, if it is differentiable and satisfies
\[
\varphi(y)\leq \varphi(x)+\langle \nabla \varphi(x), y-x\rangle+\frac{L}{2}\|y-x\|^2, \forall x, y \in\Rd.
\]
\end{definition}
\begin{definition}
\label{def: strong convexity}
We say that $\varphi: \Rd \rightarrow \R$ is $\mu-$strongly convex with $\mu\geq 0$, if 
\[
\varphi(y)\geq \varphi(x)+\langle \nabla \varphi(x), y-x\rangle+\frac{\mu}{2}\|y-x\|^2, \forall x, y \in\Rd.
\]
\end{definition}

We will make the following assumption regarding the objective \eqref{equ: original primal problem} throughout this paper.
\begin{assumption}
\label{assump: smoothness and strong convexity}
In the objective \eqref{equ: original primal problem}, each $f_{i,j}$ is $L_i-$smooth and $\mu_i-$strongly convex. Let $\kappa_i=L_i/\mu_i$,  $\mu_{\mathrm{min}}\coloneqq\min_{i} \{\mu_i\}$, $L_{\mathrm{max}}\coloneqq\max_i\{L_i\}$, $\kappa_{\mathrm{min}}\coloneqq\min_i\{\kappa_i\}$, $\kappa_{\mathrm{max}}\coloneqq\max_i\{\kappa_i\}$, and $\kappa_F\coloneqq L_{\mathrm{max}}/\mu_{\mathrm{min}}$.
\end{assumption}

In this paper, we minimize \eqref{equ: original primal problem} on a network $\mathcal{G}=\{\mathcal{V}, \mathcal{E}\}$ with $\mathcal{V}=\{1,2,...,n\}$ being the set of nodes (or workers), and $\mathcal{E}$ the set of all (undirected) edges. By convention, $(i,j)=(j,i)$ denotes the edge that connects workers $i$ and $j$. For each worker $i$, let $\mathcal{N}(i)=\{j\,|\, (i,j) \in \mathcal{E}\}\cup \{i\}$ denote the set of neighbors of worker $i$ with $i$ itself included.

In network $\mathcal{G}$, communication is represented as a matrix multiplication with a matrix $U=I-W$, where $W$ satisfies the following assumption:

\begin{assumption}
\label{assump: assumption on network topology}
\begin{enumerate}
    \item $W\in\R^{n\times n}$ is symmetric and $U=I-W$ is positive semidefinite. 
    \item $W$ is defined on the edges of network $\mathcal{G}$, that is, $W_{i,j}\neq 0$ if and only if $(i,j)\in\mathcal{E}$.
    \item $\mathrm{null}(U)=\mathrm{null}(I-W)=\mathrm{span}\{\mathbf{1}\}$.
\end{enumerate}

\end{assumption}

Let $\sigma_1(U)\geq...\geq \sigma_{n-1}(U)> \sigma_n(U)=0$ be the spectrum of $U$, and $\zeta(U)\coloneqq{\sigma_{n-1}(U)}/{\sigma_1(U)}$ as the normalized eigengap of $U$. $W$ can be generated in many ways, for example, by the maximum-degree or Metropolis-Hastings rules \cite{sayed2014adaptation}.

As mentioned before, since $\mathrm{null}(\sqrt{U})=\mathrm{span}\{\mathbf{1}\}$ and $\sqrt{U}$ is summetric, we can reformulate the problem \eqref{equ: original primal problem} as
\begin{align}
\label{equ: primal problem again}
    \mathop{\mathrm{minimize}}_{\Theta \sqrt{U}=0}F(\Theta),
\end{align}
where $\Theta=(\theta_1,\theta_2,\ldots, \theta_n)\in\RR^{d\times n}$ and $F(\Theta)=\sum\limits_{i=1}^n f_i(\theta_i)$, the condition number of $F$ is $\kappa_F$. 

The dual problem of \eqref{equ: primal problem again} can be written as
\begin{align}
\label{equ: dual problem}
    \mathop{\mathrm{minimize}}_{\xi\in\RR^{d\times n}} G(\xi)\coloneqq F^*(\xi \sqrt{U}).
\end{align}

The properties of $G$ are characterized in \cite{scaman2017optimal} as follows:
\begin{proposition}
\label{prop: properties of G}
\begin{enumerate}
    \item $G(\xi)$ is $\beta-$smooth, where $\beta\coloneqq \frac{\sigma_1(U)}{\mu_{\mathrm{min}}}$.
    \item In $S\coloneqq\{\xi \in \R^{d\times n}\,|\,\xi\mathbf{1}=0\}$,  $G(\xi)$ is $\alpha-$strongly convex, where $\alpha\coloneqq \frac{\sigma_{n-1}(U)}{L_{\mathrm{max}}}$.
    \item In $S\coloneqq\{\xi \in \R^{d\times n}\,|\,\xi\mathbf{1}=0\}$, the condition number of $G(\xi)$ is $\kappa$, where $\kappa\coloneqq\frac{\beta}{\alpha}=\frac{\kappa_F}{\zeta(U)}$.
\end{enumerate}
\end{proposition}


\section{Proposed algorithms}
\label{sec: proposed algorithm}

To solve \eqref{equ: dual problem}, \cite{scaman2017optimal} applies Nesterov's accelerated gradient descent(AGD) to the dual problem \eqref{equ: dual problem}:
\begin{align}
\label{equ: Nesterov iterations}
\begin{split}
    \lambda^{k+1}&=\xi^k-\eta\nabla F^*(\xi^k \sqrt{U})\sqrt{U},\\
    \xi^{k+1}&=\lambda^{k+1}+\frac{\sqrt{\kappa}-1}{\sqrt{\kappa}+1}(\lambda^{k+1}-\lambda^k).
\end{split}
\end{align}

With $x^k=\xi^k\sqrt{U}$ and $y^k=\lambda^k\sqrt{U}$, \eqref{equ: Nesterov iterations} simplifies to
\begin{align}
\label{equ: equivalent Nesterov iterations}
\begin{split}
    y^{k+1}&=x^k-\eta \nabla F^*(x^k)U,\\
    x^{k+1}&=y^{k+1}+\frac{\sqrt{\kappa}-1}{\sqrt{\kappa}+1}(y^{k+1}-y^k).
\end{split}
\end{align}
The authors call \eqref{equ: equivalent Nesterov iterations} SSDA. When the matrix $U$ is replaced by $P_{K}(U)$, the method is called MSDA, where $P_K$ is a polynomial of degree $K$ and $K=\lfloor1/\sqrt{\zeta(U)}\rfloor$.

SSDA and MSDA use exact dual gradients $\nabla f^*_j$, which may be expensive to obtain since $f^*_j$ is not in closed form in many applications. So, extra runtime may be required to compute an accurate dual gradient. Furthermore, at each iteration of SSDA and MSDA, every worker needs to communicate with its neighbors once or $K$ times, which leads to a lot of concurrent information exchanges.
where $|\mathcal{E}|$ is the number of edges in the network. 
This may not be feasible when the communication budget of each worker is limited.

In this work, we propose Dual Accelerated method with Lazy Approximate Gradient(DLAG) (Algorithm \ref{alg: local formulation}) and Multi-DLAG(MDLAG) (Algorithm \ref{alg: MDLAG global formulation}), where the two aforementioned issues are resolved in the following ways:

\textbf{Applying Approximate Dual Gradients of Low cost.} To get rid of the high computation cost of computing dual gradients, we propose to use approximate dual gradients that are computed efficiently. Specifically, the approximate dual gradient $\theta^k_i\approx \nabla f_i^*(x^k_i)$ is given by approximately solving the following subproblem with warm start at $\theta^{k-1}_i$:
\begin{align}
\label{equ: subproblem}
\begin{split}
\theta^k_i&\approx\argmin_{\theta\in\Rd}\{f_i(\theta)-\langle \theta, x^k_i\rangle\} \\
&=\argmin_{\theta\in\Rd}\{\frac{1}{m}\sum_{j=1}^m \left(f_{i,j}(\theta)-\langle \theta, x^k_i\rangle\right)\}.
\end{split}
\end{align}
To obtain an approximate solution $\theta^k_i$, we apply a \emph{fixed-step} subroutine, which can be taken from many algorithms, e.g.,  Nesterov's accelerated gradient descent (AGD). Other choices that exploit the finite-sum structure of \eqref{equ: subproblem} are randomized algorithms such as SVRG \cite{johnson2013accelerating} and Katyusha \cite{allen2017katyusha}. %

In previous works such as \cite{scaman2017optimal}, \cite{shen2018towards}, and \cite{hendrikx2018accelerated}, the idea of using warm start has been implemented numerically and is shown to efficient. In this work, we provide the first convergence guarantee for this strategy.

\textbf{Skipping Unnecessary Communication.} To reduce communication, we generalize the idea of lazily aggregated gradient of \cite{chen2018lag} to the decentralized setting with approximate dual gradients. Specifically, at iteration $k$, worker $i$ has $\hat{\theta}^{k-1}_i=\theta^{k-1-d^{k-1}_i}_i$ in its cache, where $d^{k-1}_i\geq 0$ is the age of the vector at iteration $k-1$. An age of 0 means ``up to date.'' If worker $i$'s lazy condition \eqref{equ: worker's condition 1} is satisfied, then worker $i$ will not send out anything to its neighbours $i'\in\mathcal{N}(i)\backslash\{i\}$, and all the workers in $\mathcal{N}(i)$ will use the \textit{lazy} approximate dual gradient $\hat{\theta}^k_i\coloneqq\hat{\theta}^{k-1}_i$ for update; If \eqref{equ: worker's condition 1} is not satisfied, then worker $i$ sets $\hat{\theta}^k_i\coloneqq\theta^k_i$ and sends $\theta^k_i-\hat{\theta}^{k-1}_i$ out to $i'\in\mathcal{N}(i)$.
\begin{align}
\textbf{Worker $i$'s lazy}& \,\textbf{condition (for skipping communication)} \nonumber\\
\|\hat{\theta}^{k-1}_i-\theta^k_i\|^2  & \leq  3\sum_{j=0}^{k-D-1}\frac{c^{k-D-j}}{\mu^2_{\mathrm{min}}}\|x^j_i-x^{j+1}_i\|^2 \nonumber\\
&\quad +3\sum_{j=0}^{k-1}\frac{c^{k-j}}{\mu^2_{\mathrm{min}}}\|x^j_i-x^{j+1}_i\|^2 \nonumber\\
&\quad+ 3\sum_{j=k-D}^{k-1}\frac{c}{\mu^2_{\mathrm{min}}}\|x^j_i-x^{j+1}_i\|^2 \nonumber\\
&\quad+3\sum_{j=k-D}^{k-1}\frac{\gamma}{\mu^2_{\mathrm{min}}}\|x^j_i-x^{j+1}_i\|^2. \label{equ: worker's condition 1}
\end{align}
In \eqref{equ: worker's condition 1}, $c\in (0,1)$ controls inexactness of the approximate dual gradient, a larger $c$ requires less accurate dual gradients but causes a larger iteration complexity; $\gamma>0$ reflects the tolerance for gradient staleness, a larger $\gamma$ leads to less frequent communication but more iterations. Finally, $D$ is the maximal delay of gradients and we enforce $d^k_i\leq D$ for all workers and iterations. In Sec. \ref{sec: main theory}, we will show that appropriate choices of $c$ and $\gamma$ 
lead to computation and communication reduction. 
\begin{remark}
\begin{enumerate}
    \item The lazy condition \eqref{equ: worker's condition 1} adapts to the data heterogeneity across all the workers. We will see in Theorem \ref{thm: communication complexity} that, workers with smaller smoothness constants $\frac{1}{\mu_i}$ satisfy their lazy conditions more often, thus can skip more communication. 
    \item The lazy condition \eqref{equ: worker's condition 1} can be implemented with a mild memory requirement of $\mathcal{O}(D)$.
\end{enumerate}
\end{remark}

\begin{algorithm}[ht]
\caption{Dual Decentralized learning with Lazy and Approximate dual gradients (DLAG)}
\label{alg: local formulation}
    \textbf{Input:} $x^0_i=y^0_i=0$, $\hat{\theta}^0_i=\theta^0_i=\nabla f^*_i(x^0_i)$\footnotemark[1], and $P^0_i=\sum_{j\in\mathcal{N}(i)}U_{i j}\theta^0_j$, step size $\eta>0$, parameter $s\geq 1$.
    \begin{flushleft}
    \textbf{Output:} 
    $y^K=(y^K_1, y^K_2, ...,y^K_n).$
    \end{flushleft}
    \begin{algorithmic}[1]
        \FOR{each worker $i$ in parallel}
        \STATE{Read $P_i^{k-1}$, $\hat{\theta}^{k-1}_i$, and $\theta^{k-1}_i$ from cache;} 
        \STATE{Get $\theta^k_i$ via {\small $\mathcal{O}\left((m+\sqrt{m\kappa_{\mathrm{max}}})\log(\frac{2\kappa_{\mathrm{max}}}{c})\right)$} stochastic gradient steps of Katyusha, warm started at $\theta^{k-1}_i$;}
        \IF{$\hat{\theta}^{k-1}_i$
        fails condition \eqref{equ: worker's condition 1} \textbf{or} $d^{k-1}_i=D$}
        \STATE{Send $Q_i^k\coloneqq\theta^k_i-\hat{\theta}^{k-1}_{i}$ to worker $i'\in\mathcal{N}(i)\backslash \{i\}$;}
        \STATE{$\hat{\theta}^{k}_{i}\leftarrow {\theta}^{k}_{i}$;}
        \STATE{$d^{k}_i=0$;}
        \ELSE{}
        \STATE{(Worker $i$ sends out nothing)}
        \STATE{$\hat{\theta}^{k}_i\leftarrow \hat{\theta}^{k-1}_i$;}
        \STATE{$d^k_i=d^{k-1}_i+1$;}
        \ENDIF
        \STATE{Let $S^k_i\coloneqq\{j\in\mathcal{N}(i)\mid j \,\,\text{sends out}\,\, Q^k_j\}$;}
        \STATE{Update cache: $P_i^k\leftarrow P_i^{k-1}+\sum_{j\in S^k_i }U_{ij}Q_j^k$;}
        \STATE{$y^{k+1}_i\leftarrow x^k_i-\eta P_i^k$;}
        \STATE{$x^{k+1}_i\leftarrow y^{k+1}_i+\frac{\sqrt{s\kappa}-1}{\sqrt{s\kappa}+1}(y^{k+1}_i-y^k_i)$;}
        \ENDFOR
        \STATE{$k\leftarrow k+1$;}
    \end{algorithmic}
\end{algorithm}
\footnotetext[1]{DLAG computes exact dual gradient only \textit{once} for initialization. This cost is negligible. %
}

We can also formulate DLAG in an equivalent form using matrix multiplications in Algorithm \ref{alg: global formulation}, which makes it easy to present our theoretical analyses.

\begin{algorithm}[ht]
\caption{DLAG: global formulation}
\label{alg: global formulation}
    \textbf{Input:} problem data $F(\Theta)=\sum f_i(\theta_i)$, initialization $x^0=y^0=0$ and $\hat{\Theta}^0=\Theta^0=\nabla F^*(x^0)$ step size $\eta>0$, parameter $s\geq 1$.
    \begin{flushleft}
    \textbf{Output:} $y^K=(y^K_1, y^K_2, ...,y^K_n).$
    \end{flushleft}
    \begin{algorithmic}[1]
        \STATE{$y^{k+1}\leftarrow x^k-\eta \hat{\Theta}^k U$;}
        \STATE{$x^{k+1}\leftarrow y^{k+1}+\frac{\sqrt{s\kappa}-1}{\sqrt{s\kappa}+1}(y^{k+1}-y^k)$;}
        \STATE{$k\leftarrow k+1$;}
    \end{algorithmic}
\end{algorithm}
In line 1 of Algorithm \ref{alg: global formulation},  $\hat{\Theta}^k=(\hat{\theta}^k_1,...,\hat{\theta}^k_n)$. Note that $\hat{\theta}^k_i={\theta}^{k}_i$ if worker $i$'s lazy condition \eqref{equ: worker's condition 1} is unsatisfied or the delay $d_{k-1}=D$, and $\hat{\theta}^k_i=\hat{\theta}^{k-1}_i$ otherwise. 

Following the same idea, we also apply the idea of lazy and approximate dual gradients to MSDA, which gives MDLAG (Algorithm \ref{alg: MDLAG global formulation}). Compared with DLAG, MDLAG applies the Chebyshev acceleration technique \cite{young2014iterative}, where the gossip matrix $U$ is replaced by $P_K(U)$ and $P_K$ is a polynomial of power $K$\footnotemark[1]. $P_K(U)$ requires $K$ rounds of communication. MDLAG has a better computation complexity but may not save communication as much as DLAG since it only reduces the communication of the first round. The detail of MDLAG can be found in App. \ref{app: MDLAG}.
\footnotetext[1]{Specifically, $P_K(U)=I-\frac{T_K(c_2(I-U))}{T_K(c_2I)},$ where $c_2=\frac{1+\gamma}{1-\gamma}$ and $T_K$ is a Chebyshev polynomial of power $K$.}
\begin{algorithm}[ht]
\caption{MDLAG: global formulation}
\label{alg: MDLAG global formulation}
    \textbf{Input:} problem data $F(\Theta)=\sum f_i(\theta_i)$, initialization $x^0=y^0=0$ and $\hat{\Theta}^0=\Theta^0=\nabla F^*(x^0)$ step size $\eta>0$, parameter $s\geq 1$, $K=\lfloor\frac{1}{\sqrt{\zeta(U)}}\rfloor, \kappa'=\frac{\kappa_F}{\zeta(P_K(U))}.$
    \begin{flushleft}
    \textbf{Output:} $y^K=(y^K_1, y^K_2, ...,y^K_n)$
    \end{flushleft}
    \begin{algorithmic}[1]
        \STATE{$y^{k+1}\leftarrow x^k-\eta \hat{\Theta}^k P_K(U)$;}
        \STATE{$x^{k+1}\leftarrow y^{k+1}+\frac{\sqrt{s\kappa'}-1}{\sqrt{s\kappa'}+1}(y^{k+1}-y^k)$;}
        \STATE{$k\leftarrow k+1$;}
    \end{algorithmic}
\end{algorithm}

\section{Main theory}
\label{sec: main theory}

In this section, we proceed to establish the gradient and communication complexities of Algorithms \ref{alg: global formulation} and \ref{alg: MDLAG global formulation}. %

First for DLAG, note that the iterations $x^k, y^k$ in Algorithm \ref{alg: global formulation} satisfy $x^k\mathbf{1}=y^k\mathbf{1}=0$, so there exist $\xi^k, \lambda^k$ such that $x^k=\xi^k\sqrt{U}, y^k=\lambda^k\sqrt{U}$ ($\xi^k$ and $\lambda^k$ are never calculated in practice, they are just for the purpose of proof). Algorithm \ref{alg: global formulation} can then be written as
\begin{align}
\label{equ: equivalent global formulation}
\begin{split}
    \lambda^{k+1}&=\xi^k-\eta \hat{\Theta}^k \sqrt{U},\\
    \xi^{k+1}&=\lambda^{k+1}+\frac{\sqrt{s\kappa}-1}{\sqrt{s\kappa}+1}(\lambda^{k+1}-\lambda^k).
\end{split}
\end{align}
Comparing \eqref{equ: equivalent global formulation} with SSDA \eqref{equ: Nesterov iterations}, we can see that their difference lies in the update directions $\hat{\Theta}^k\sqrt{U}$ and $\nabla F^*({\xi}^{k}\sqrt{U})\sqrt{U}$. To bound their difference, we need the following lemma characterizing the dynamics of the error $\Theta^k-\nabla F^*(x^k)$, which lays the theoretical foundation for incorporating the warm start into our convergence proof.

\begin{lemma}[Error propagating dynamics]
\label{lem: error propagating dynamics}
In line 2 of Algorithm \ref{alg: local formulation}, 
if the subproblem \eqref{equ: subproblem} is solved by Katyusha with $\mathcal{O}\left((m+\sqrt{m\kappa_{\mathrm{max}}})\log(\frac{2\kappa_{\mathrm{max}}}{c})\right)=\mathcal{O}(m+\sqrt{m\kappa_{\mathrm{max}}})$ stochastic gradient evaluations warm started at $\theta^{k-1}_i$, then
\[
\mathbb{E}\|\Theta^k-\nabla F^*(x^k)\|^2\leq \sum_{j=0}^{k-1}{c}^{k-j}\mathbb{E}\|\nabla F^*(x^j)-\nabla F^*(x^{j+1})\|^2.
\]
\end{lemma}

\begin{proof}
See Appendix \ref{app: proof of error dynamics}.
\end{proof}

In view of Lemma \ref{lem: error propagating dynamics}, we introduce the following Lyapunov function for Algorithm \ref{alg: global formulation} to establish convergence.
\begin{align}
\label{equ: lyapunov function}
&L^k\coloneqq 2\eta s\kappa\left(G(\lambda^k)-G(\xi^{\star})\right)+\|v^k-\xi^{\star}\|^2\nonumber\\ 
&+\sum\limits_{d=1}^{D}c_d\|\xi^{k+1-d}-\xi^{k-d}\|^2+\sum_{d=1}^k\tilde{c}_d\|\xi^{k+1-d}-\xi^{k-d}\|^2.
\end{align}
Here $\xi^{\star}$ is the minimizer of $G(\xi)$, $v^k=(1+\sqrt{s\kappa})\xi^k-\sqrt{s\kappa}\lambda^k$, the constants $c_d>0$ and $\tilde{c}_d>0$ will be specified. %

In \eqref{equ: lyapunov function}, the last two terms are introduced to deal with the error in \eqref{equ: gradient error}.  With $c_d=\tilde{c}_d=0$, $s=1$, and $\eta=\frac{1}{\beta}$, \eqref{equ: lyapunov function} reduces to the Lyapunov function proposed in \cite{wilson2016lyapunov} for AGD \eqref{equ: Nesterov iterations}.

\begin{theorem}[Stochastic gradient complexity of DLAG]
\label{thm: iteration complexity}
Take Assumptions \ref{assump: smoothness and strong convexity} and \ref{assump: assumption on network topology}.
Take $\gamma=\frac{\alpha\beta\mu^2_{\mathrm{min}}}{288D\|\sqrt{U}\|^4}e^{-\frac{2D}{\sqrt{\kappa}}}$, $c=\frac{\alpha\beta\mu^2_{\mathrm{min}}}{1200 D\|\sqrt{U}\|^4}e^{-\frac{2(D+1)}{\sqrt{\kappa}}}<1$, $\eta=\frac{2}{15}\frac{1}{\beta}$ and $s = 10$, where $\kappa = \frac{\kappa_F}{\zeta_U}$. At each iteration, apply Katyusha with $\mathcal{O}(m+\sqrt{m\kappa_{\mathrm{max}}})$ stochastic gradient evaluations and warm start. Then, we have $\E [L^{k+1}]\leq \left(1-\frac{1}{\sqrt{10\kappa}}\right)\E[L^k]$ for any $k\geq 0$. In order to obtain an approximate solution to \eqref{equ: dual problem} with $\varepsilon-$suboptimality, DLAG needs an iteration complexity of
\begin{align*}
\mathcal{I}_{\mathrm{MDLAG}}(\varepsilon)=\mathcal{O}\left(\sqrt{\kappa}\log(\frac{1}{\varepsilon})\right).
\end{align*}
and a stochastic gradient complexity of
\begin{align*}
\mathcal{G}_{\mathrm{MDLAG}}(\varepsilon)=\mathcal{O}\left(n(m+\sqrt{m\kappa_{\mathrm{max}}})\sqrt{\kappa}\log(\frac{1}{\varepsilon})\right).
\end{align*}
\end{theorem}
\begin{proof}
See Appendix \ref{app: iteration complexity}.
\end{proof}
Compared with DLAG, MDLAG applies a better gossip matrix $P_K(U)$. Because of this, MDLAG enjoys a better computation complexity\footnotemark[1]. 
\footnotetext[1]{Essentially, $K=\lfloor\frac{1}{\sqrt{\zeta(U)}}\rfloor$ guarantees that $P_K(U)$ has a large normalized eigengap of $\zeta(P_K(U))\geq\frac{1}{4}$.}
\begin{theorem}[Stochastic gradient complexity of MDLAG]
\label{thm: MDLAG iteration complexity}
Take Assumptions \ref{assump: smoothness and strong convexity} and \ref{assump: assumption on network topology}.
Take $\gamma=\frac{\alpha'\beta'\mu^2_{\mathrm{min}}}{288D\|\sqrt{P_K(U)}\|^4}e^{-\frac{2D}{\sqrt{\kappa_F}}}$, $c=\frac{\alpha'\beta'\mu^2_{\mathrm{min}}}{1200 D\|\sqrt{P_K(U)}\|^4}e^{-\frac{2(D+1)}{\sqrt{\kappa_F}}}<1$, $\eta=\frac{2}{15}\frac{1}{\beta'}$ and $s = 10$, where $\alpha' = \frac{\sigma_{n-1}(P_K(U))}{L_{\max}}$ and $\beta'=\frac{\sigma_1(P_K(U))}{\mu_{\min}}$. At each iteration, apply Katyusha with $\mathcal{O}(m+\sqrt{m\kappa_{\mathrm{max}}})$ stochastic gradient evaluation and warm start. Then, 
in order to obtain an $\varepsilon-$suboptimal solution to \eqref{equ: dual problem}, MDLAG needs an iteration complexity of
\begin{align}
\label{equ: iteration complexity}
\mathcal{I}_{\mathrm{MDLAG}}(\varepsilon)=\mathcal{O}\left(\sqrt{\kappa_F}\log(\frac{1}{\varepsilon})\right)
\end{align}
in expectation, and a stochastic gradient complexity of
\begin{align}
\label{equ: primal gradient complexity by Katyusha}
\mathcal{G}_{\mathrm{MDLAG}}(\varepsilon)=\mathcal{O}\left(n(m+\sqrt{m\kappa_{\mathrm{max}}})\sqrt{\kappa_F}\log(\frac{1}{\varepsilon})\right).
\end{align}
\end{theorem}
\begin{proof}
See Appendix \ref{app: MDLAG}.
\end{proof}
Next, we provide the communication complexity of Algorithms \ref{alg: global formulation} and \ref{alg: MDLAG global formulation}. First, we define the heterogeneity score function: %
\begin{align}
\label{equ: score function}
h_d(\gamma)=\frac{1}{2|\mathcal{E}|}\sum\limits_{i=1}^n m_i\mathbb{1}\left(H^2_i\leq\frac{\gamma}{d}\right), \  d=1,2,...,D.
\end{align}
Here, $H_i\coloneqq\frac{1/\mu_i}{1/\mu_{\mathrm{min}}}=\frac{\mu_{\mathrm{min}}}{\mu_i}$ is the importance factor of worker $i$ (recall $f^*_i$ is $\frac{1}{\mu_i}-$smooth), $m_i$ is the number of edges connected to worker $i$, $|\mathcal{E}|$ is the total number of edges in network, and $\mathbb{1}$ equals $1$ if $H_i^2\leq\frac{\gamma}{d}$ and $0$ otherwise.

For each $d$, $h_d(\gamma)\in [0,1]$ reflects the percentages of edges that are connected to a worker $i$ with importance factor smaller or equal to $\frac{\gamma}{d}$. In our context, $h_d(\gamma)$ critically lower bounds the fractions of direct edges where communication happens at most $\frac{k}{d+1}$ times until the $k$th iteration.

\begin{theorem}[Communication complexity of DLAG and MDLAG]
Take the assumptions of Theorem \ref{thm: iteration complexity}. In order to obtain an approximate solution to \eqref{equ: dual problem} with $\varepsilon-$suboptimality, Algorithm \ref{alg: global formulation} and \ref{alg: MDLAG global formulation} have communication complexities of
\label{thm: communication complexity}
\begin{align*}
\begin{split}
&\mathcal{C}_{\mathrm{DLAG}}(\varepsilon)\leq\Big( 1 -\sum\limits_{d=1}^D\left(\frac{1}{d}-\frac{1}{d+1}\right)h_d(\gamma)\Big)2|\mathcal{E}|\mathcal{I}_{\mathrm{DLAG}}(\varepsilon)\\
&\mathcal{C}_{\mathrm{MDLAG}}(\varepsilon)\leq\Big(K-\sum\limits_{d=1}^D\left(\frac{1}{d}-\frac{1}{d+1}\right)h_d(\gamma)\Big)2|\mathcal{E}|\mathcal{I}_{\mathrm{MDLAG}}(\varepsilon).
\end{split}
\end{align*}
\end{theorem}
\begin{proof}
See Appendix \ref{app: communication complexity}.
\end{proof}

\begin{remark}
\label{rem: communication complexity}
If $\gamma=0$, then $\mathcal{C}_{\mathrm{DLAG}}(\varepsilon)$ reduces to SSDA's communication complexity $\mathcal{C}_{SSDA}(\varepsilon)=2|\mathcal{E}|\mathcal{I}_{SSDA}(\varepsilon)$, and $\mathcal{C}_{\mathrm{MDLAG}}(\varepsilon)$ reduces to $\mathcal{C}_{MSDA}(\vareps)=2K|\mathcal{E}|\mathcal{I}_{MSDA}(\varepsilon).$
\end{remark}

\begin{corollary}
\label{coro: compare communication complexity}
Under the settings of Theorem \ref{thm: communication complexity}, we have
\begin{align}
\label{equ: compare communication complexity}
\frac{\mathcal{C}_{\mathrm{DLAG}}(\varepsilon)}{\mathcal{C}_{SSDA}(\varepsilon)}\leq q\coloneqq \sqrt{10}\left(1-\sum\limits_{d=1}^D\left(\frac{1}{d}-\frac{1}{d+1}\right)h_d\left(\gamma\right)\right).
\end{align}
\end{corollary}

From \eqref{equ: score function} we know that, if there are a large fraction of workers with big $\mu_i$, then $q$ is much smaller than $1$, and DLAG can save a lot of communication compared with SSDA. An illustrative example can be found at Appendix \ref{app: communication complexity}. MDLAG does not save as much communication since it needs $K-1$ full rounds of communication at each iteration.

\section{Experiments}
\label{sec: experiments}

In this section, we compare our DLAG and MDLAG with state-of-the-art decentralized algorithms\footnotemark[1]: COLA \cite{he2018cola}, SSDA, and MSDA %
on the \texttt{heart} dataset from LIBSVM\footnotemark[2].

\footnotetext[1]{Comparison with ADFS \cite{hendrikx2019accelerated} can be found in App. \ref{app: compare with ADFS}.}
\footnotetext[2]{\url{https://www.csie.ntu.edu.tw/ cjlin/libsvmtools/datasets/}}
We formulate cross-entropy minimization as
\begin{equation*}
    \min_{x\in\Rd} \frac{1}{n_0}\sum_{i=1}^{n_0} \big(-b_i\log\sigma(a_i^Tx)-(1-b_i)\log\sigma(-a_i^Tx)\big) +\lambda\|x\|^2,
\end{equation*}
where $A_0=(a_1,a_2,...,a_{n_0})^T\in\R^{n_0\times d}$, $\lambda=0.01$, and $\sigma(z)=\frac{1}{1+e^{-z}}$.

\begin{enumerate}
    \item The decentralized network is 5x5 2D grid.
    \item Data is unevenly distributed on the network. Theoretically, this leads to smaller importance factors, thus more communication save (see \eqref{equ: score function} and Theorem \ref{thm: communication complexity}). Specifically, for $b>a>0$, we first generate $p_i\sim\text{rand}[a,b], i=1,\ldots, n$. The number of data samples on worker $i$ is proportional to $p_i/(\sum_{j=1}^n p_j)$.
    \item For SSDA and DLAG, stepsize is $\eta = 1/\beta$, for MSDA and MDLAG, stepsize is $\eta=1/\beta'$, where $\beta=\frac{\sigma_1(U)}{\mu_{\min}}$ and $\beta'=\frac{\sigma_1(P_K(U))}{\mu_{\min}}$. We set $s=1$, $\gamma=c=1e-4$ and $D=50$ for our DLAG and MDLAG.
    \item For CoLa, we set the aggregation parameter to be 1, and apply 40 epochs of Nesterov's accelerated gradient descent to solve the local subproblem. 
\end{enumerate}

For cross-entropy minimization, the dual gradient is not immediately available, so we apply 30 epochs of Katyusha to obtain an approximate dual gradient in DLAG and MDLAG. For SSDA and MSDA, it is unclear how accurate the approximate dual gradients should be. We apply Katyusha to solve the subproblem \eqref{equ: subproblem} until reaching an accuracy of $1e-10$. This benefits SSDA and MSDA since by \cite{uribe2018dual}, one should actually apply Katyusha until reaching an accuracy of $\vareps^2=1e-14$ to guarantee overall convergence, where $\vareps=1e-7$ is the final target accuracy.

 \begin{figure}[H]
 \centering
 \includegraphics[width=0.38\textwidth]{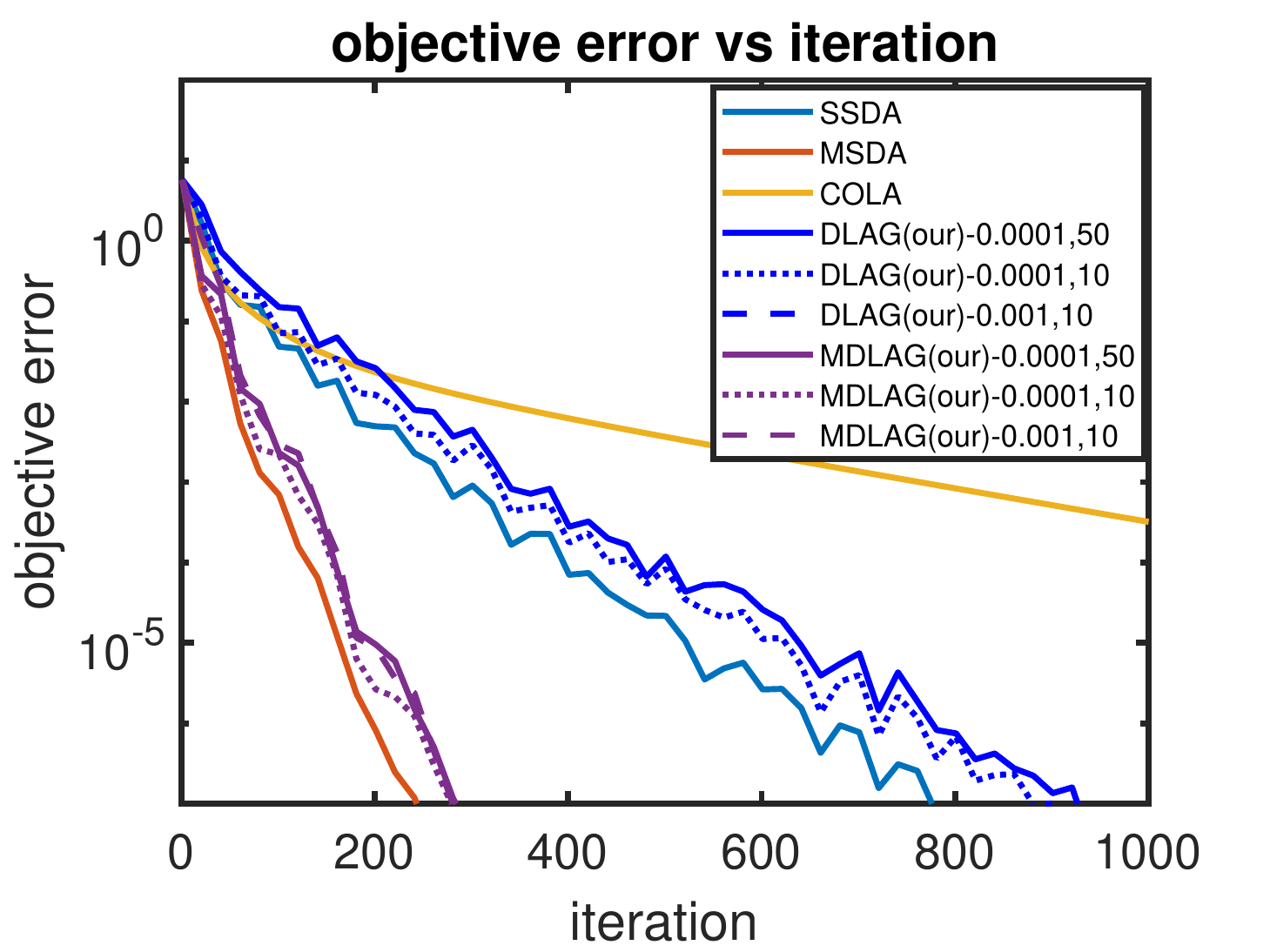}
 \caption{Iteration complexities on \texttt{heart} dataset. DLAG-$0.0001,10$ means DLAG with $\gamma=c=0.0001$ and $D=10$.}
 \label{fig: heart_iter}
 \end{figure}
 \begin{figure}[ht]
 \centering
 \includegraphics[width=0.38\textwidth]{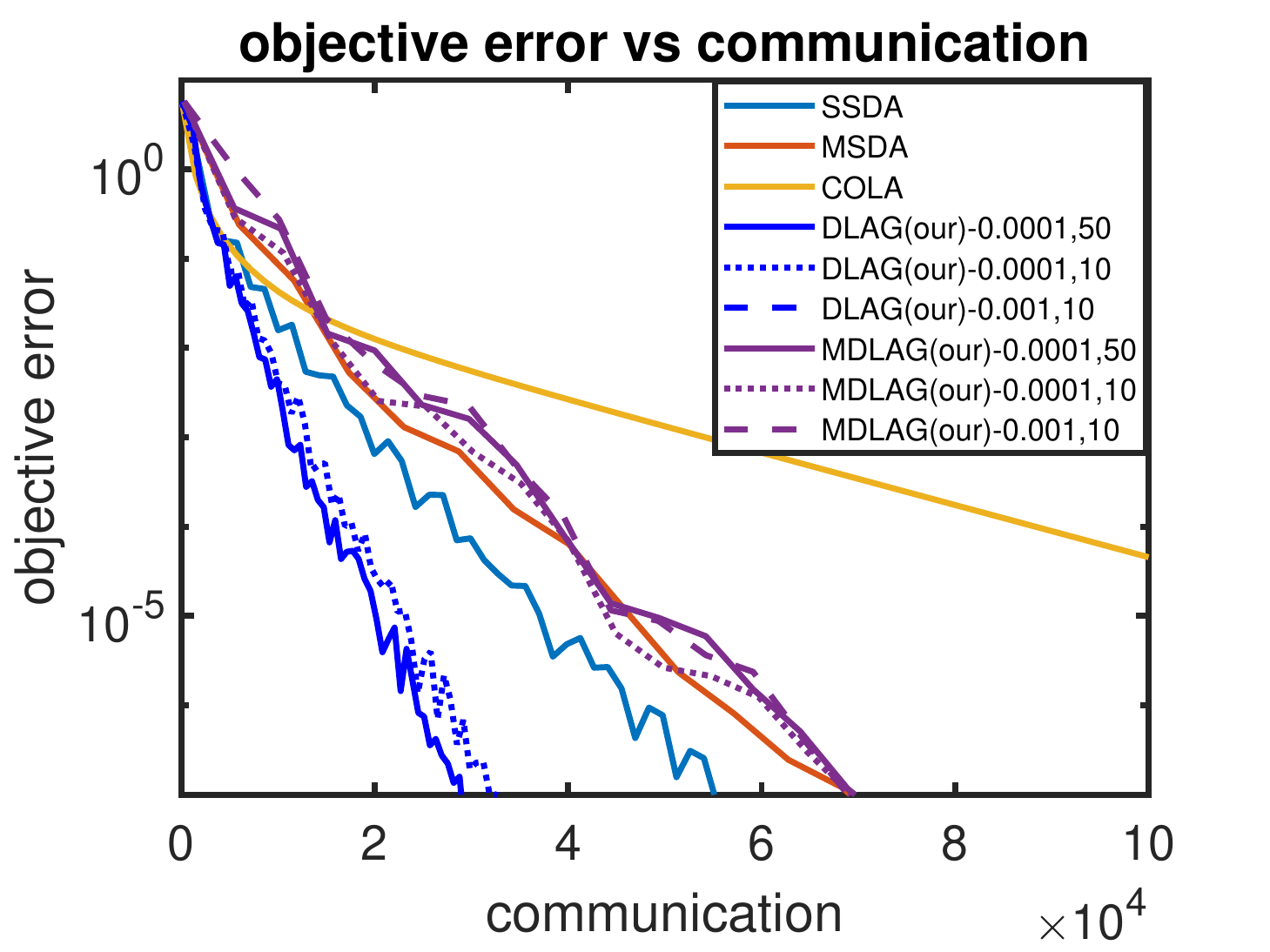}
 \caption{Communication complexities on \texttt{heart} dataset.}
 \label{fig: heart_comm}
 \end{figure}
 \begin{figure}[ht]
 \centering
 \includegraphics[width=0.38\textwidth]{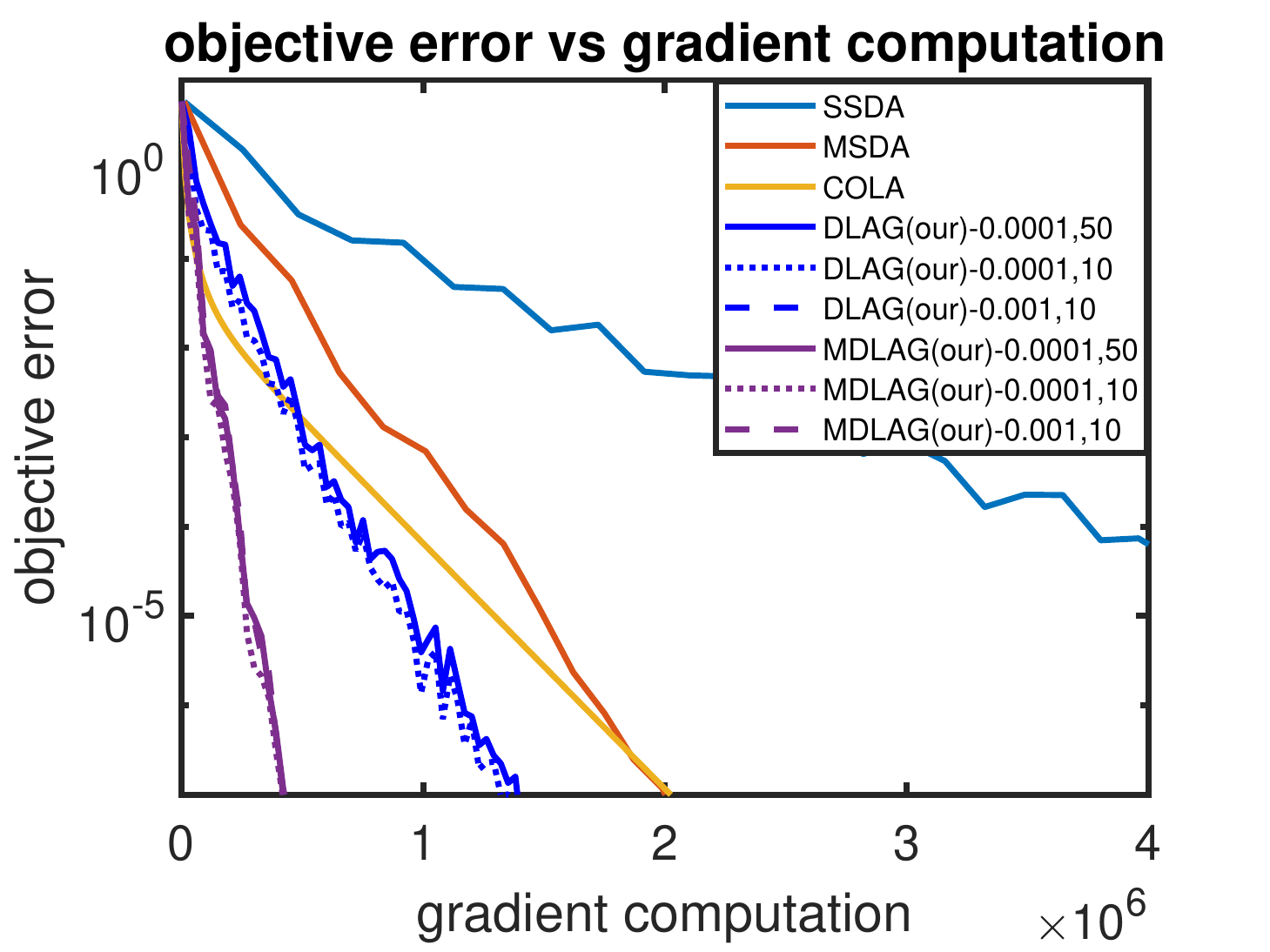}
 \caption{\footnotesize Stochastic gradient complexities on \texttt{heart} dataset.}
 \label{fig: heart_gradient}
 \end{figure}

From Figures \ref{fig: heart_iter}, \ref{fig: heart_comm}, and \ref{fig: heart_gradient}, we can see that the behaviors of tested algorithms match Table \ref{table: comparison}.
\begin{enumerate}
    \item The performance of DLAG and MDLAG is robust to the choice of parameters. 
    \item DLAG and MDLAG achieve iteration complexities similar to those of SSDA and MSDA, respectively.
    \item DLAG still uses the least communication (about $40\%$ less than SSDA). %
    \item MDLAG has the smallest gradient complexity (about $80\%$ less than MSDA).
\end{enumerate}

\section{Conclusions and future work}

In this work, we propose DLAG and MDLAG for decentralized machine learning, where computation is saved by applying highly approximate dual gradients, and unnecessary communication can be skipped based on a dynamic criterion. Compared with other methods, DLAG does not rely on extra oracles to compute exact dual gradients or proximal mappings, and successfully reduces communication complexity. All these claims are justified numerically.

There are still open problems to be addressed. For example, can we also apply the worker's lazy condition for all $K$ rounds of communication in MDLAG, so that it can enjoy least amount of computation and communication?

%
%
%



\appendix

\subsection{Proof of Lemma \ref{lem: error propagating dynamics}: error propagating dynamics}
\label{app: proof of error dynamics}

\begin{proof}
The inexact dual gradient $\theta^k_i$ is produced by solving the following subproblem by Katyusha:
\begin{align}
\label{equ: subproblem again}
\begin{split}
\theta^k_i &\approx  \argmin_{\theta\in\Rd}\{\frac{1}{m}\sum_{j=1}^m \left(f_{i,j}(\theta)-\langle \theta, x^k_i\rangle\right)\}.
\end{split}
\end{align}
We will apply Katyusha such that the inexact dual gradient $\theta^k_i$ satisfies
\begin{align}
\label{equ: Katyusha error}
\mathbb{E}_k\|\theta^k_{i}-\nabla f^*_i(x^k_i)\|^2\leq \frac{c}{2}\|\theta^{k-1}_{i}-\nabla f^*_i(x^k_i)\|^2,
\end{align}
where $\nabla f^*_i(x^k_i)$ is the solution of \eqref{equ: subproblem again}.

Denote $F_i(\theta)=f_i(\theta)-\langle\theta, x^k_i\rangle$. By Theorem 3.1 of \cite{allen2017katyusha}, we know that if Katyusha is warm started at $\theta^{k-1}_i$, 
Katyusha needs 
$$\mathcal{O}\left((m+\sqrt{m \kappa_i})\log(\frac{F_i(\theta^{k-1}_i)-F_i(\nabla f^*_i(x^k_i))}{\vareps_0})\right)$$ 
stochastic gradient evaluations in expectation to reach 
\begin{align}
\label{equ: objective error}
\mathbb{E}_{k}[F_i(\theta^k_i)-F_i(\nabla f^*_i(x^k_i))]\leq \vareps_0.
\end{align}
Here, if we take 
\begin{align}
\label{equ: obj err'}
\vareps_0 = \frac{\mu_i c}{4}\|\theta^{k-1}_i-\nabla f^*_i(x^k_i)\|^2.
\end{align}
then we obtain a stochastic gradient complexity of 
\begin{align*}
&\mathcal{O}\left((m+\sqrt{m \kappa_i})\log(\frac{4}{\mu_i c}\frac{F_i(\theta^{k-1}_i)-F_i(\nabla f^*_i(x^k_i))}{\|\theta^{k-1}_i-\nabla f^*_i(x^k_i)\|^2})\right)\\
&=\mathcal{O}\left((m+\sqrt{m \kappa_i})\log(\frac{2\kappa_{\mathrm{max}}}{c})\right).
\end{align*}
On the other hand, from \eqref{equ: objective error} and \eqref{equ: obj err'} we have 
$$\mathbb{E}_{k}\|\theta^k_i-\nabla f^*_i(x^k_i)\|^2\leq \frac{c}{2}\|\theta^{k-1}_i-\nabla f^*_i(x^k_i)\|^2,$$
which is exactly \eqref{equ: Katyusha error}.
Furthermore, \eqref{equ: Katyusha error} leads to
\begin{align}
\label{equ: temp}
\begin{split}
\mathbb{E}_k\|\theta^k_i-\nabla f^*_i(x^k_i)\|^2 &\leq c\|\theta^{k-1}_{i}-\nabla f^*_i(x^{k-1}_i)\|^2 \\
&\quad + c\|\nabla f^*_i(x^{k}_i)-\nabla f^*_i(x^{k-1}_i)\|^2.
\end{split}
\end{align}
Define $a^k_i = \E\|\theta^k_i-\nabla f^*_i(x^k_i)\|^2$ and $b^k_i=\E\|\nabla f^*_i(x^{k}_i)-\nabla f^*_i(x^{k-1}_i)\|^2$. Then, \eqref{equ: temp} becomes the following recursion:
\[
\frac{a_i^k}{c^k}-\frac{a^{k-1}_i}{c^{k-1}}\leq \frac{b^{k-1}_i}{c^{k-1}},
\]
Since $a^0_i=\|\theta^0_i-\nabla f^*_i(x^0_i)\|^2=0$, we have $\frac{a^k_i}{c^k}\leq \sum_{j=0}^{k-1}\frac{b^j_i}{c^j}$, or equivalently,
\begin{align}
\label{equ: error1}
\E\|\theta^k_i-\nabla f^*_i(x^k_i)\|^2\leq \sum_{j=0}^{k-1}c^{k-j}\E\|\nabla f^*_i(x^{j}_i)-\nabla f^*_i(x^{j+1}_i)\|^2.
\end{align}
The desired result follows.
%
\end{proof}

\subsection{Gradient error bound}
\label{app: proof of gradient error}
In this section, we prove a lemma on the gradient error.

\begin{lemma}[Gradient error]
\label{lem: gradient error}
Under the same settings as in Lemma \ref{lem: error propagating dynamics}, the difference between $\hat{\Theta}^k\sqrt{U}$ and the true gradient $g^k\coloneqq\nabla F^*({\xi}^{k}\sqrt{U})\sqrt{U}$ satisfies:
\begin{align}
\label{equ: gradient error}
\begin{split}
&\E\|\hat{\Theta}^k\sqrt{U}-g^k\|^2\\
&\leq 6\|\sqrt{U}\|^4\sum_{j=0}^{k-D-1}\frac{c^{k-D-j}}{\mu^2_{\mathrm{min}}}\E\|\xi^j_i-\xi^{j+1}_i\|^2\\
&\quad+8\|\sqrt{U}\|^4\sum_{j=0}^{k-1}\frac{c^{k-j}}{\mu^2_{\mathrm{min}}}\E\|\xi^j-\xi^{j+1}\|^2\\
&\quad+ 6\|\sqrt{U}\|^4\sum_{j=k-D}^{k-1}\frac{c}{\mu^2_{\mathrm{min}}}\E\|\xi^j-\xi^{j+1}\|^2\\
&\quad+6\|\sqrt{U}\|^4\sum_{j=k-D}^{k-1}\frac{\gamma}{\mu^2_{\mathrm{min}}}\E\|\xi^j-\xi^{j+1}\|^2,
\end{split}
\end{align}
\end{lemma}

\begin{proof}

First of all, we have 
\begin{align}
\label{equ: bounding gradient error}
\|\hat{\Theta}^k\sqrt{U}-g^k\|^2&\leq 2\|(\hat{\Theta}^k-\Theta^k)\sqrt{U}\|^2\nonumber\\
&\quad+2\|\Theta^k\sqrt{U}-\nabla F^*(x^k)\sqrt{U}\|^2.
\end{align}
In DLAG, if worker $i$'s lazy condition \eqref{equ: worker's condition 1} is satisfied, then $\hat{\theta}^k_i=\hat{\theta}^{k-1}_i$ (skipping communication), and $\hat{\theta}^k_i={\theta}^{k}_i$ (perform communication) otherwise. As a result, we have
\begin{align*}
\begin{split}
\|\hat{\theta}^{k}_i-\theta^k_i\|^2  & \leq  3\sum_{j=0}^{k-D-1}\frac{c^{k-D-j}}{\mu^2_{\mathrm{min}}}\|x^j_i-x^{j+1}_i\|^2\\
&\quad +3\sum_{j=0}^{k-1}\frac{c^{k-j}}{\mu^2_{\mathrm{min}}}\|x^j_i-x^{j+1}_i\|^2\\
&\quad + 3\sum_{j=k-D}^{k-1}\frac{c}{\mu^2_{\mathrm{min}}}\|x^j_i-x^{j+1}_i\|^2\\
&\quad +3\sum_{j=k-D}^{k-1}\frac{\gamma}{\mu^2_{\mathrm{min}}}\|x^j_i-x^{j+1}_i\|^2.
\end{split}
\end{align*}
In view of this, the first term on the right hand side of \eqref{equ: bounding gradient error} can be then bounded as
\begin{align}
\label{equ: bounding the first term}
\begin{split}
&2\|(\hat{\Theta}^k-\Theta^k)\sqrt{U}\|^2\\
&\leq 2\|\sqrt{U}\|^2\|\hat{\Theta}^k-\Theta^k\|^2\\
&\leq 6\|\sqrt{U}\|^4\sum_{j=0}^{k-D-1}\frac{c^{k-D-j}}{\mu^2_{\mathrm{min}}}\|\xi^j_i-\xi^{j+1}_i\|^2\\
&\quad+6\|\sqrt{U}\|^4\sum_{j=0}^{k-1}\frac{c^{k-j}}{\mu^2_{\mathrm{min}}}\|\xi^j-\xi^{j+1}\|^2\\
&\quad+ 6\|\sqrt{U}\|^4\sum_{j=k-D}^{k-1}\frac{c}{\mu^2_{\mathrm{min}}}\|\xi^j-\xi^{j+1}\|^2\\
&\quad +6\|\sqrt{U}\|^4\sum_{j=k-D}^{k-1}\frac{\gamma}{\mu^2_{\mathrm{min}}}\|\xi^j-\xi^{j+1}\|^2,
\end{split}
\end{align}
where we have applied $x=\xi\sqrt{U}$ in the second inequality.

To bound the second term on the right hand side of \eqref{equ: bounding gradient error}, we can apply Lemma \ref{lem: error propagating dynamics} in the following way:
\begin{align}
\label{equ: bounding the second term}
\begin{split}
&2\|\Theta^k\sqrt{U}-\nabla F^*(x^k)\sqrt{U}\|^2\\
&\leq 2\|\sqrt{U}\|^2\sum_{j=0}^{k-1}c^{k-j}\E\|\nabla F^*(x^j)-\nabla F^*(x^{j+1})\|^2\\
&\leq 2\|\sqrt{U}\|^2\sum_{j=0}^{k-1}\frac{c^{k-j}}{\mu^2_{\mathrm{min}}}\E\|x^j-x^{j+1}\|^2\\
&\leq 2\|\sqrt{U}\|^4\sum_{j=0}^{k-1}\frac{c^{k-j}}{\mu^2_{\mathrm{min}}}\E\|\xi^j-\xi^{j+1}\|^2
\end{split}
\end{align}
where we have applied the $\frac{1}{\mu_{\mathrm{min}}}-$smoothness of $F^*$ in the first inequality, and $x=\xi\sqrt{U}$ in the second inequality. 

Finally, combining \eqref{equ: bounding gradient error}, \eqref{equ: bounding the first term}, and \eqref{equ: bounding the second term} yields the desired result.
%
%
%
%
%
%
%
%
%
%
%
%
%
%
%
%
%
%
%
%
%
%
\end{proof}
\subsection{Preliminary propositions}
\label{app: Preliminary propositions}
First, let us define $\Delta v^k \coloneqq v^k-\xi^{\star}$ and $\Delta\xi^k \coloneqq\xi^k-\xi^{\star}.$
Then, the Lyapunov function $L^k$ in \eqref{equ: lyapunov function} can be written as %
\begin{align*}
L^k=2\eta s\kappa\left(G(\lambda^k)-G(\xi^{\star})\right)+\|\Delta v^k\|^2+A^k+\tilde{A}^k,
\end{align*}
where 
\begin{align}
    v^k&=\xi^k+\sqrt{s\kappa}(\xi^k-\lambda^k),\label{equ: def of v}\\
    A^k&=\sum\limits_{d=1}^{D}c_d\|\xi^{k+1-d}-\xi^{k-d}\|^2,\label{equ: def of A^k}\\
    \tilde{A}^k&=\sum\limits_{d=1}^{k}\tilde{c}_d\|\xi^{k+1-d}-\xi^{k-d}\|^2.\label{equ: def of tilde A^k}
\end{align}
We want to obtain $L^{k+1}-(1-\frac{1}{\sqrt{s\kappa}}) L^k\leq 0$ for some $s\geq 1$. For this purpose, we bound the terms in $L^{k+1}$ in the following propositions. Their proofs can be found in Appendices \ref{app: proof of G bound}, \ref{app: proof of v bound}, and \ref{app: proof of xi bound}, respectively.

\begin{proposition}
\label{prop: bound G}
We have
\begin{align*}
G(\lambda^{k+1})&\leq G(\xi^k)-\left(\eta-\frac{\eta^2\beta}{2}-\frac{\eta^2\sqrt{s\kappa}\alpha}{2\rho}\right)\|\hat{\Theta}^k\sqrt{U}\|^2\\
& \quad+\frac{\rho}{2\sqrt{s\kappa}\alpha}\|\hat{\Theta}^k\sqrt{U}-g^k\|^2.
\end{align*}
\end{proposition}

\begin{proposition}
We have
\label{prop: bound v}
\begin{align*}
&\|\Delta v^{k+1}\|^2\\
&\leq\left(1-\frac{1}{\sqrt{s\kappa}}\right)\|\Delta v^k\|^2\\
&\ \ \ +\left(\frac{1}{\sqrt{s\kappa}}-\eta\alpha\sqrt{s\kappa}\left(1-\frac{1}{\rho}\right)\right)\|\xi^k-\xi^{\star}\|^2\\
    &\ \ \ \ +\left(1-\frac{1}{\sqrt{s\kappa}}\right)\left(-\frac{1}{\sqrt{s\kappa}}+\eta\alpha\left(\frac{\sqrt{s\kappa}-1}{\rho}-1\right)\right)\|v^k-\xi^k\|^2\\
    &\ \ \ \ +2\eta s\kappa\left(\left(G(\xi^{\star})-G(\xi^k)\right)+(1-\frac{1}{\sqrt{s\kappa}})\left(G(\lambda^k)-G(\xi^{\star})\right)\right)\\
    &\ \ \ \ +\eta^2s\kappa\|\hat{\Theta}^k\sqrt{U}\|^2+2\rho\frac{\eta\sqrt{s\kappa}}{\alpha}\|\hat{\Theta}^k\sqrt{U}-g^k\|^2.
\end{align*}
\end{proposition}

\begin{proposition}
\label{prop: bound xi}
Let $c_d=\sum\limits_{j=d}^{D}(1-\frac{1}{\sqrt{s\kappa}})^{d-j-1}q$ and $q>0$ for $d=1,2,...,D$. Then, we have $c_{d+1}-(1-\frac{1}{\sqrt{s\kappa}})c_d = -q$ and
\begin{align*}
A^{k+1}-(1-\frac{1}{\sqrt{s\kappa}})A^k
&\leq 2c_1\left(\frac{\sqrt{s\kappa}-1}{\sqrt{s\kappa}(\sqrt{s\kappa}+1)}\right)^2\|v^k-\xi^k\|^2\\
&\quad +2c_1\left(\frac{2\eta\sqrt{s\kappa}}{\sqrt{s\kappa}+1}\right)^2\|\hat{\Theta}^k\sqrt{U}\|^2\\
&\ \ \ \ -\sum\limits_{d=1}^{D}q\|\xi^{k+1-d}-\xi^{k-d}\|^2.
\end{align*}
Similarly, we have
\begin{align*}
&\tilde{A}^{k+1}-(1-\frac{1}{\sqrt{s\kappa}})\tilde{A}^k\\
&\leq 2\tilde{c}_1\left(\frac{\sqrt{s\kappa}-1}{\sqrt{s\kappa}(\sqrt{s\kappa}+1)}\right)^2\|v^k-\xi^k\|^2\\
&\quad+2\tilde{c}_1\left(\frac{2\eta\sqrt{s\kappa}}{\sqrt{s\kappa}+1}\right)^2\|\hat{\Theta}^k\sqrt{U}\|^2\\
&\ \ \ \ +\sum\limits_{d=1}^{k}(\tilde{c}_{d+1}-(1-\frac{1}{\sqrt{s\kappa}})\tilde{c}_d)\|\xi^{k+1-d}-\xi^{k-d}\|^2.
\end{align*}
\end{proposition}

\subsubsection{A toolkit for proof}
Before diving into the details of proof of our main theory, let us first list some useful equalities and inequalities.

We will use $g^k=\nabla F^*(\xi^k\sqrt{U})\sqrt{U}$ throughout the rest of the proof.

\begin{enumerate}
    \item For $v^k$ defined in \eqref{equ: def of v}, we have
    \begin{align}
        v^{k+1}&=(1+\sqrt{s\kappa})\xi^{k+1}-\sqrt{s\kappa}\lambda^{k+1} \nonumber\\
        &=(1+\sqrt{s\kappa})\left(\lambda^{k+1}+\frac{\sqrt{s\kappa}-1}{\sqrt{s\kappa}+1}(\lambda^{k+1}-\lambda^k)\right) \nonumber\\
        &\quad-\sqrt{s\kappa}\lambda^{k+1} \nonumber\\
        &=\sqrt{s\kappa}\lambda^{k+1}-(\sqrt{s\kappa}-1)\lambda^k \nonumber\\
        &=\sqrt{s\kappa}(\xi^k-\eta\hat{\Theta}^k\sqrt{U})\nonumber\\
        &\quad -(\sqrt{s\kappa}-1)\left(\left(1+\frac{1}{\sqrt{s\kappa}}\right)\xi^k-\frac{1}{\sqrt{s\kappa}}v^k\right) \nonumber\\
        &=\left(1-\frac{1}{\sqrt{s\kappa}}\right)v^k+\frac{1}{\sqrt{s\kappa}}\xi^k-\eta\sqrt{s\kappa}\hat{\Theta}^k\sqrt{U}.\label{equ: v}
    \end{align}
    \item (Young's inequality) For any $a, b\in\RR$ and $\chi>0$, we have 
    \begin{align}\label{equ: young}
        ab\leq\frac{\chi a^2}{2}+\frac{b^2}{2\chi}.
    \end{align}
    \item By Proposition \ref{prop: properties of G}, for any $\xi_1, \xi_2\in\R^{d\times n}$, we have 
    \begin{align}
        G(\xi_2)
        &\leq G(\xi_1)+\dotp{\nabla G(\xi_1), \xi_2-\xi_1}+\frac{\beta}{2}\|\xi_2-\xi_1\|^2, \label{equ: smooth}\\
        G(\xi_2)&\geq G(\xi_1)+\dotp{\nabla G(\xi_1), \xi_2-\xi_1}+\frac{\alpha}{2}\|\xi_2-\xi_1\|^2. \label{equ: strongly convex}
    \end{align}
    \item For any $0\leq r\leq 1$ and $x,y\in\Rn$, we have
    \begin{align}
    \label{equ: parallelogram}
        \|(1-r)x+r y\|^2&=(1-r)\|x\|^2+r\|y\|^2\nonumber\\
        &\quad-r(1-r)\|x-y\|^2.
    \end{align}
    \item For any $0\leq x\leq\frac{1}{2},y\geq 0$, we have 
    \begin{align}
    \label{equ: to exp}
    (1-x)^{-y}\leq e^{2xy}.
    \end{align}
\end{enumerate}

\subsubsection{Proof of Proposition \ref{prop: bound G}}
\label{app: proof of G bound}
We have
\begin{align*}
    &G(\lambda^{k+1})=G(\xi^k-\eta\hat{\Theta}^k\sqrt{U})\\
    &\overset{(a)}{\leq} G(\xi^k)-\eta\dotp{g^k,\hat{\Theta}^k\sqrt{U}}+\frac{\eta^2\beta}{2}\|\hat{\Theta}^k\sqrt{U}\|^2\\
    &=G(\xi^k)-\eta\|\hat{\Theta}^k\sqrt{U}\|^2+\eta\dotp{\hat{\Theta}^k\sqrt{U}-g^k,\hat{\Theta}^k\sqrt{U}}\\
    &\quad +\frac{\eta^2\beta}{2}\|\hat{\Theta}^k\sqrt{U}\|^2\\
    &\overset{(b)}{\leq} G(\xi^k)-\left(\eta-\frac{\eta^2\beta}{2}-\frac{\eta^2\sqrt{s\kappa}\alpha}{2\rho}\right)\|\hat{\Theta}^k\sqrt{U}\|^2\\
    &\quad +\frac{\rho}{2\sqrt{s\kappa}\alpha}\|\hat{\Theta}^k\sqrt{U}-g^k\|^2,
\end{align*}
where $(a)$ follows from the smoothness of $G$ in \eqref{equ: smooth}, $(b)$ follows from \eqref{equ: young} with $\chi=\frac{\rho}{\eta\sqrt{s\kappa}\alpha}$, and $\rho>0$ will be determined later.

\subsubsection{Proof of Proposition \ref{prop: bound v}}
\label{app: proof of v bound}

Equation \eqref{equ: def of v} implies that
\begin{align*}
    \Delta v^{k+1}=\left(1-\frac{1}{\sqrt{s\kappa}}\right)\Delta v^k+\frac{1}{\sqrt{s\kappa}}\Delta\xi^k-\eta\sqrt{s\kappa}\hat{\Theta}^k\sqrt{U}.
\end{align*}
Therefore,
\begin{align*}
    &\|\Delta v^{k+1}\|^2\\
    &=\left\|\left(1-\frac{1}{\sqrt{s\kappa}}\right)\Delta v^k+\frac{1}{\sqrt{s\kappa}}\Delta\xi^k\right\|^2+\eta^2s\kappa\|\hat{\Theta}^k\sqrt{U}\|^2\\
    &\quad-2\eta\sqrt{s\kappa}\left\langle\left(1-\frac{1}{\sqrt{s\kappa}}\right)\Delta v^k+\frac{1}{\sqrt{s\kappa}}\Delta\xi^k, \hat{\Theta}^k\sqrt{U}\right\rangle\\
    &\overset{(a)}{=}\left(1-\frac{1}{\sqrt{s\kappa}}\right)\|\Delta v^k\|^2+\frac{1}{\sqrt{s\kappa}}\|\Delta\xi^k\|^2\\
    &\quad -\left(1-\frac{1}{\sqrt{s\kappa}}\right)\frac{1}{\sqrt{s\kappa}}\|v^k-\xi^k\|^2+\eta^2s\kappa\|\hat{\Theta}^k\sqrt{U}\|^2\\
    &\ \ \ -2\eta\sqrt{s\kappa}\left\langle\left(1-\frac{1}{\sqrt{s\kappa}}\right)\Delta v^k+\frac{1}{\sqrt{s\kappa}}\Delta\xi^k, \hat{\Theta}^k\sqrt{U}\right\rangle\\
    &\overset{(b)}{=}\left(1-\frac{1}{\sqrt{s\kappa}}\right)\|\Delta v^k\|^2+\frac{1}{\sqrt{s\kappa}}\|\Delta\xi^k\|^2\\
    &\quad -\left(1-\frac{1}{\sqrt{s\kappa}}\right)\frac{1}{\sqrt{s\kappa}}\|v^k-\xi^k\|^2+\eta^2s\kappa\|\hat{\Theta}^k\sqrt{U}\|^2\\
    &\ \ \ -2\eta\sqrt{s\kappa}\dotp{\Delta \xi^k+(\sqrt{s\kappa}-1)(\xi^k-\lambda^k), \hat{\Theta}^k\sqrt{U}}\\
\end{align*}
where $(a)$ follows from \eqref{equ: parallelogram} and $\Delta v^k-\Delta \xi^k = v^k-\xi^k$, $(b)$ follows from the definition of $v^k$ in \eqref{equ: def of v}.

For $-\dotp{\Delta \xi^k, \hat{\Theta}^k\sqrt{U}}$ we have
\begin{align*}
    &-\dotp{\Delta\xi^k,\hat{\Theta}^k\sqrt{U}}\\
    &=-\dotp{\Delta\xi^k, g^k}-\dotp{\Delta\xi^k, \hat{\Theta}^k\sqrt{U}-g^k}\\
    &\overset{(c)}{\leq} G(\xi^{\star})-G(\xi^k)-\frac{\alpha}{2}\|\Delta\xi^k\|^2-\dotp{\Delta\xi^k, \hat{\Theta}^k\sqrt{U}-g^k}\\
    &\overset{(d)}{\leq} G(\xi^{\star})-G(\xi^k)-\frac{\alpha}{2}\|\Delta\xi^k\|^2+\frac{\rho}{2\alpha}\|\hat{\Theta}^k\sqrt{U}-g^k\|^2\\
    &\quad +\frac{\alpha}{2\rho}\|\Delta\xi^k\|^2.\\
\end{align*}
where (c) follows from the strong convexity of $G$ in \eqref{equ: strongly convex}, and (d) follows from Young's inequality \eqref{equ: young} with $\chi=\frac{\rho}{\alpha}$. 

Similarly, for $-\dotp{\xi^k-\lambda^k,\hat{\Theta}^k\sqrt{U}}$ we have
\begin{align*}
    &-\dotp{\xi^k-\lambda^k,\hat{\Theta}^k\sqrt{U}}\\
    &=-\dotp{\xi^k-\lambda^k,g^k}-\dotp{\xi^k-\lambda^k,\hat{\Theta}^k\sqrt{U}-g^k}\\
    &\leq G(\lambda^k)-G(\xi^k)-\frac{\alpha}{2}\|\xi^k-\lambda^k\|^2\\
    &\quad +\frac{\rho}{2(\sqrt{s\kappa}-1)\alpha}\|\hat{\Theta}^k\sqrt{U}-g^k\|^2\\
    &\quad +\frac{(\sqrt{s\kappa}-1)\alpha}{2\rho}\|\xi^k-\lambda^k\|^2\\
    &=G(\lambda^k)-G(\xi^k)\\
    &\quad-\frac{1}{s\kappa}\left(\frac{\alpha}{2}-\frac{(\sqrt{s\kappa}-1)\alpha}{2\rho}\right)\|v^k-\xi^k\|^2\\
    &\quad +\frac{\rho}{2(\sqrt{s\kappa}-1)\alpha}\|\hat{\Theta}^k\sqrt{U}-g^k\|^2.
\end{align*}
where the last equality follows from \eqref{equ: def of v}.

As a result,
\begin{align*}
    &\|\Delta v^{k+1}\|^2\\
    &\leq\left(1-\frac{1}{\sqrt{s\kappa}}\right)\|\Delta v^k\|^2+\frac{1}{\sqrt{s\kappa}}\|\Delta\xi^k\|^2\\
    &\quad -\left(1-\frac{1}{\sqrt{s\kappa}}\right)\frac{1}{\sqrt{s\kappa}}\|v^k-\xi^k\|^2+\eta^2s\kappa\|\hat{\Theta}^k\sqrt{U}\|^2\\
    &\ \ \ \  +2\eta\sqrt{s\kappa}\bigg(G(\xi^{\star})-G(\xi^k)-\frac{\alpha}{2}\left(1-\frac{1}{\rho}\right)\|\Delta\xi^k\|^2\\ &\quad\quad\quad\quad\quad +\frac{\rho}{2\alpha}\|\hat{\Theta}^k\sqrt{U}-g^k\|^2\bigg)\\
    &\ \ \ \ +2\eta\sqrt{s\kappa}(\sqrt{s\kappa}-1)\\
    &\quad\times\bigg(G(\lambda^k)-G(\xi^k)
    -\frac{1}{s\kappa}\frac{\alpha}{2}\left(1-\frac{\sqrt{s\kappa}-1}{\rho}\right)\|v^k-\xi^k\|^2\\
    &\quad\quad\quad+\frac{\rho}{2(\sqrt{s\kappa}-1)\alpha}\|\hat{\Theta}^k\sqrt{U}-g^k\|^2\bigg)\\
    &=\left(1-\frac{1}{\sqrt{s\kappa}}\right)\|\Delta v^k\|^2+\left(\frac{1}{\sqrt{s\kappa}}-\eta\alpha\sqrt{s\kappa}\left(1-\frac{1}{\rho}\right)\right)\|\Delta\xi^k\|^2\\
    &+\left(1-\frac{1}{\sqrt{s\kappa}}\right)\left(-\frac{1}{\sqrt{s\kappa}}+\eta\alpha\left(\frac{\sqrt{s\kappa}-1}{\rho}-1\right)\right)\|v^k-\xi^k\|^2\\
    &+2\eta s\kappa\left(\left(G(\xi^{\star})-G(\xi^k)\right)+(1-\frac{1}{\sqrt{s\kappa}})\left(G(\lambda^k)-G(\xi^{\star})\right)\right)\\
    &+\eta^2s\kappa\|\hat{\Theta}^k\sqrt{U}\|^2+2\rho\frac{\eta\sqrt{s\kappa}}{\alpha}\|\hat{\Theta}^k\sqrt{U}-g^k\|^2.
\end{align*}

\subsubsection{Proof of Proposition \ref{prop: bound xi}}
\label{app: proof of xi bound}

Since $c_d=\sum\limits_{j=d}^{D}r^{d-j-1}q$, where $r= 1-\frac{1}{\sqrt{s\kappa}}$ and $q>0$, we have
\begin{align*}
    &A_{k+1}-r A_k\\
    &=c_1\|\xi^{k+1}-\xi^k\|^2+\sum\limits_{d=1}^{D-1}(c_{d+1}-r c_d)\|\xi^{k+1-d}-\xi^{k-d}\|^2\\
    &\quad-r c_{D}\|\xi^{k+1-D}-\xi^{k-D}\|^2\\
    &=c_1\|\xi^{k+1}-\xi^k\|^2-\sum\limits_{d=1}^{D}q\|\xi^{k+1-d}-\xi^{k-d}\|^2.\\
\end{align*}

To deal with $\xi^{k+1}-\xi^k$, we can write
\begin{align*}
    &\xi^{k+1}-\xi^k\\
    &=\lambda^{k+1}-\xi^k+\frac{\sqrt{s\kappa}-1}{\sqrt{s\kappa}+1}(\lambda^{k+1}-\xi^k+\xi^k-\lambda^k)\\
    &=-\eta\hat{\Theta}^k\sqrt{U}+\frac{\sqrt{s\kappa}-1}{\sqrt{s\kappa}+1}(-\eta\hat{\Theta}^k\sqrt{U})\\
    &\quad+\frac{\sqrt{s\kappa}-1}{\sqrt{s\kappa}+1}\frac{1}{\sqrt{s\kappa}}(v^k-\xi^k)\\
    &=\frac{1}{1+\sqrt{s\kappa}}\left((1-\frac{1}{\sqrt{s\kappa}})(v^k-\xi^k)-2\eta\sqrt{s\kappa}\hat{\Theta}^k\sqrt{U}\right).
\end{align*}
Therefore,
\begin{align*}
    &A_{k+1}-rA_k\\
    &= c_1\left\|\frac{1}{1+\sqrt{s\kappa}}\left((1-\frac{1}{\sqrt{s\kappa}})(v^k-\xi^k)-2\eta\sqrt{s\kappa}\hat{\Theta}^k\sqrt{U}\right)\right\|^2\\
    &\quad-\sum\limits_{d=1}^{D}q\|\xi^{k+1-d}-\xi^{k-d}\|^2\\
    &\leq 2c_1\left(\frac{\sqrt{s\kappa}-1}{\sqrt{s\kappa}(\sqrt{s\kappa}+1)}\right)^2\|v^k-\xi^k\|^2\\
    &\quad +2c_1\left(\frac{2\eta\sqrt{s\kappa}}{\sqrt{s\kappa}+1}\right)^2\|\hat{\Theta}^k\sqrt{U}\|^2-\sum\limits_{d=1}^{D}q\|\xi^{k+1-d}-\xi^{k-d}\|^2.
\end{align*}
Similarly, we also have
\begin{align*}
    &\tilde{A}_{k+1}-r\tilde{A}_k\\
    &=
    \tilde{c}_1\|\xi^{k+1}-\xi^k\|^2+\sum\limits_{d=1}^{k}(\tilde{c}_{d+1}-r\tilde{c}_d)\|\xi^{k+1-d}-\xi^{k-d}\|^2\\
    &= \tilde{c}_1\left\|\frac{1}{1+\sqrt{s\kappa}}\left((1-\frac{1}{\sqrt{s\kappa}})(v^k-\xi^k)-2\eta\sqrt{s\kappa}\hat{\Theta}^k\sqrt{U}\right)\right\|^2\\
    &\quad+\sum\limits_{d=1}^{k}(\tilde{c}_{d+1}-r\tilde{c}_d)\|\xi^{k+1-d}-\xi^{k-d}\|^2\\
    &\leq 2\tilde{c}_1\left(\frac{\sqrt{s\kappa}-1}{\sqrt{s\kappa}(\sqrt{s\kappa}+1)}\right)^2\|v^k-\xi^k\|^2\\
    &\quad+2c_1\left(\frac{2\eta\sqrt{s\kappa}}{\sqrt{s\kappa}+1}\right)^2\|\hat{\Theta}^k\sqrt{U}\|^2\\
    &\quad+\sum\limits_{d=1}^{k}(\tilde{c}_{d+1}-r\tilde{c}_d)\|\xi^{k+1-d}-\xi^{k-d}\|^2.
\end{align*}

\subsection{Stochastic Gradient complexity}
\label{app: iteration complexity}
In order to prove Theorem \ref{thm: iteration complexity} directly follows, we prove a slightly more general result stated in Theorem \ref{thm: iteration complexity full}.
\begin{theorem}
\label{thm: iteration complexity full}
Take Assumptions \ref{assump: smoothness and strong convexity} and \ref{assump: assumption on network topology}, and let $\sqrt{\kappa_{\mathrm{max}}}\log(\frac{128\kappa^3_{\mathrm{max}}}{c})$, and 
\[\eta=\frac{\sqrt{2+\frac{1}{12(a+b)}}}{1+\sqrt{2+\frac{1}{12(a+b)}}}\frac{1}{1+24(a+b)(2+\sqrt{2+\frac{1}{12(a+b)}})}\frac{1}{\beta},
\] 
\[s = \frac{2+\sqrt{2+\frac{1}{12(a+b)}}}{\sqrt{2+\frac{1}{12(a+b)}}}\left(1+24(a+b)\left(2+\sqrt{2+\frac{1}{12(a+b)}}\right)\right),
\] 
where $a=\frac{6\|\sqrt{U}\|^4}{\alpha\beta\mu^2_{\mathrm{min}}}\gamma D e^{\frac{2D}{\sqrt{\kappa}}}$ and $ b=\frac{25\|\sqrt{U}\|^4}{\alpha\beta\mu^2_{\mathrm{min}}}c D e^{\frac{2D}{\sqrt{\kappa}}}$, where $\kappa = \frac{\kappa_F}{\zeta_U}$. Assume that $\kappa>2$. At each iteration, let the subproblem \eqref{equ: subproblem} be solved by Katyusha with $\mathcal{O}(m+\sqrt{m\kappa_{\mathrm{max}}})$ stochastic gradient evaluations. Then, we have $\E[L^{k+1}]\leq \left(1-\frac{1}{\sqrt{s\kappa}}\right)\E[L^k]$ for any $k\geq 0$. Therefore, in order to obtain an approximate solution to \eqref{equ: dual problem} with $\varepsilon-$suboptimality, Algorithm \ref{alg: global formulation} has an iteration complexity of
\begin{align*}
\mathcal{I}_{\mathrm{DLAG}}(\varepsilon)=\mathcal{O}\left(\sqrt{s\kappa}\log(\frac{1}{\varepsilon})\right),
\end{align*}
in expectation, and a stochastic gradient complexity of 
\begin{align*}
\mathcal{G}_{\mathrm{DLAG}}(\varepsilon)=\mathcal{O}\left((m+\sqrt{m\kappa_{\mathrm{max}}})\sqrt{s\kappa}\log(\frac{1}{\varepsilon})\right).
\end{align*}
in expectation.
\end{theorem}

\subsection{Proof of Theorem \ref{thm: iteration complexity full}}
Combining Propositions \ref{prop: bound G}, \ref{prop: bound v}, and \ref{prop: bound xi} with the definition of $L^k$ in \eqref{equ: lyapunov function} yields 
\begin{align*}
\begin{split}
    &\ \ \ L^{k+1}-\left(1-\frac{1}{\sqrt{s\kappa}}\right)L^k\\
    &\leq\|\Delta\xi^k\|^2\times C_1+\|v^k-\xi^k\|^2\times C_2\\
    &\ \ \ +\|\hat{\Theta}^k\sqrt{U}\|^2\times C_3\\
    &\ \ \ +\|\hat{\Theta}^k\sqrt{U}-g^k\|^2\times (3\rho \sqrt{s\kappa}\eta\frac{1}{\alpha})\\
    &\ \ \ +\sum\limits_{d=1}^{D}\|\xi^{k+1-d}-\xi^{k-d}\|^2\times\left(-q\right)\\
    &\ \ \ +\sum\limits_{d=1}^{k}\|\xi^{k+1-d}-\xi^{k-d}\|^2\times\left(\tilde{c}_{d+1}-r\tilde{c}_d\right)
\end{split}
\end{align*}
where the coefficients $C_1, C_2$, and $C_3$ are
\begin{align}
\label{equ: coefficients}
\begin{split}
C_1 &= \left(\frac{1}{\sqrt{s\kappa}}-\eta\alpha\sqrt{s\kappa}\left(1-\frac{1}{\rho}\right)\right),\\
C_2&= \Bigg(\left(1-\frac{1}{\sqrt{s\kappa}}\right)\bigg(-\frac{1}{\sqrt{s\kappa}}+\eta\alpha\big(\frac{\sqrt{s\kappa}-1}{\rho}-1\big)\bigg)\\
    &\quad\quad\quad\quad\quad\quad\quad\quad\quad\quad\quad+2c_1\left(\frac{\sqrt{s\kappa}-1}{\sqrt{s\kappa}(\sqrt{s\kappa}+1)}\right)^2\\
    &\quad\quad\quad\quad\quad\quad\quad\quad\quad\quad\quad+2\tilde{c}_1\left(\frac{\sqrt{s\kappa}-1}{\sqrt{s\kappa}(\sqrt{s\kappa}+1)}\right)^2\Bigg),\\
C_3 &= \eta^2s\kappa\Bigg(-1+2c_1\left(\frac{2}{\sqrt{s\kappa}+1}\right)^2+2\tilde{c}_1\left(\frac{2}{\sqrt{s\kappa}+1}\right)^2\\
    &\quad\quad\quad\quad\quad\quad\quad\quad\quad\quad\quad+\left(\eta\beta+\frac{\eta\alpha\sqrt{s\kappa}}{\rho}\right)\Bigg).
\end{split}
\end{align}
By Lemma \ref{lem: gradient error} we further know that  
\begin{align}
\label{equ: E bound}
\begin{split}
    &\ \ \ \E[L^{k+1}-\left(1-\frac{1}{\sqrt{s\kappa}}\right)L^k]\\
    &\leq C_1 \E\|\Delta\xi^k\|^2+ C_2 \E\|v^k-\xi^k\|^2+C_3 \E\|\hat{\Theta}^k\sqrt{U}\|^2\\
    &\ \ \ +\left(\frac{3\rho\sqrt{s\kappa}\eta}{\alpha}6\|\sqrt{U}\|^4\frac{\gamma}{\mu^2_{\mathrm{min}}}-q\right)\sum\limits_{d=1}^{D}\E\|\xi^{k+1-d}-\xi^{k-d}\|^2\\
    &\ \ \ +\left(\frac{3\rho\sqrt{s\kappa}\eta}{\alpha}6\|\sqrt{U}\|^4\frac{c}{\mu^2_{\mathrm{min}}}\right)\sum\limits_{d=1}^{D}\E\|\xi^{k+1-d}-\xi^{k-d}\|^2\\
    &\ \ \ +\left(\frac{3\rho\sqrt{s\kappa}\eta}{\alpha}8\|\sqrt{U}\|^4\frac{c^d}{\mu^2_{\mathrm{min}}}\right)\sum\limits_{d=1}^{k}\E\|\xi^{k+1-d}-\xi^{k-d}\|^2\\
    &\ \ \ +\left(\frac{3\rho\sqrt{s\kappa}\eta}{\alpha}6\|\sqrt{U}\|^4\frac{c^{d-D}}{\mu^2_{\mathrm{min}}}\right)\sum\limits_{d=D+1}^{k}\E\|\xi^{k+1-d}-\xi^{k-d}\|^2\\
    &\ \ \ +\left(\tilde{c}_{d+1}-(1-\frac{1}{\sqrt{s\kappa}})\tilde{c}_d\right)\sum\limits_{d=1}^{k}\E\|\xi^{k+1-d}-\xi^{k-d}\|^2.
\end{split}
\end{align}

In the rest of the proof, we will select $\eta, \rho, s$, and $q$ such that the right-hand side of \eqref{equ: E bound} is non-positive. Therefore, $L^k$ converges to $0$ at a linear rate of $1-\frac{1}{\sqrt{s\kappa}}$. Recalling the definition of $L^k$ in \eqref{equ: lyapunov function}, we know that the (expected) iteration complexity for Algorithm \ref{alg: global formulation} to obtain an $\varepsilon-$suboptimal solution is  
\begin{align*}
    \mathcal{I}_{\mathrm{DLAG}}=\mathcal{O}\left(\sqrt{s\kappa}\log(\frac{1}{\varepsilon})\right),
\end{align*}
where $s$ will be specified in \eqref{equ: sss}.

\subsubsection{Bound \texorpdfstring{$c_1$}{} to make the first and fourth term of \texorpdfstring{\eqref{equ: E bound}}{} non-positive}

In order for the coefficient $C_1$ in \eqref{equ: E bound} and \eqref{equ: coefficients} to be non-negative, we can set $\rho>1$ and 
\begin{align*}
    \eta s\beta\geq\frac{1}{1-\frac{1}{\rho}}=\frac{\rho}{\rho-1}.
\end{align*}
Therefore, it suffices to set 
\begin{align}
\label{equ: s}
s=\frac{\rho}{\rho-1}\frac{1}{\eta\beta}.
\end{align}

Let us also set $q=\frac{18\rho\eta\|\sqrt{U}\|^4}{\alpha\mu^2_{\mathrm{min}}}\sqrt{s\kappa}\gamma$ to make the fourth term of \eqref{equ: E bound} to be $0$. Therefore,
\begin{align}
\label{equ: c_1}
\begin{split}
    c_1&=\sum\limits_{j=1}^{D}(1-\frac{1}{\sqrt{s\kappa}})^{-j}q\leq q D(1-\frac{1}{\sqrt{s\kappa}})^{-D}\\
    &\leq q D e^{\frac{2D}{\sqrt{s\kappa}}}\leq q D e^{\frac{2D}{\sqrt{\kappa}}}\\
    &=\frac{18\rho\eta\|\sqrt{U}\|^4}{\alpha\mu^2_{\mathrm{min}}}\sqrt{s\kappa}\gamma D e^{\frac{2D}{\sqrt{\kappa}}}= 3 a \rho\eta\beta\sqrt{s\kappa}.
\end{split}
\end{align}
where the second inequality follows from \eqref{equ: to exp} (note that $\kappa>2$ and $s\geq 1$). In the last equality, we have set 
\begin{align}
\label{equ: a}
a= \frac{6\|\sqrt{U}\|^4}{\alpha\beta\mu^2_{\mathrm{min}}}\gamma D e^{\frac{2D}{\sqrt{\kappa}}}.
\end{align} %

\subsubsection{Bound \texorpdfstring{$\tilde{c}_1$}{} to make the sum of last 4 terms of \texorpdfstring{\eqref{equ: E bound}}{} non-positive}

Let $r=1-\frac{1}{\sqrt{s\kappa}}$. In order to make the sum of the last four terms to be non-positive, we require that
\begin{align*}
&\frac{\tilde{c}_d}{r^d}-\frac{\tilde{c}_{d+1}}{r^{d+1}}\\
&=\left( 6\frac{c}{r^{d+1}}+8\frac{c^d}{r^{d+1}} \right)\frac{\|\sqrt{U}\|^4}{\mu^2_{\mathrm{min}}}(3\rho\sqrt{s\kappa}\eta\frac{1}{\alpha})\quad \text{for}\,\, 1\leq d\leq D,\\
&\frac{\tilde{c}_d}{r^d}-\frac{\tilde{c}_{d+1}}{r^{d+1}}\\
&=\left( 6\frac{c^{d-D}}{r^{d+1}}+8\frac{c^d}{r^{d+1}} \right)\frac{\|\sqrt{U}\|^4}{\mu^2_{\mathrm{min}}}(3\rho\sqrt{s\kappa}\eta\frac{1}{\alpha})\quad \text{for}\,\, d\geq D+1.
\end{align*}
In order to ensure that $\tilde{c}_d>0$ for all $d\geq 1$, let us take $\tilde{c}_1$ such that
\begin{align*}
\frac{\tilde{c}_1}{r}&=\sum_{d=D+1}^{\infty} \left( 6\frac{c^{d-D}}{r^{d+1}}+8\frac{c^d}{r^{d+1}} \right)\frac{\|\sqrt{U}\|^4}{\mu^2_{\mathrm{min}}}(3\rho\sqrt{s\kappa}\eta\frac{1}{\alpha})\\
&\ \ \ + \sum_{d=1}^D \left( 6\frac{c}{r^{d+1}}+8\frac{c^d}{r^{d+1}} \right)\frac{\|\sqrt{U}\|^4}{\mu^2_{\mathrm{min}}}(3\rho\sqrt{s\kappa}\eta\frac{1}{\alpha}).
\end{align*}
Let $c\leq\frac{r}{2}$, we have
\begin{align*}
\tilde{c}_1&=\left(6\frac{\frac{c}{r^{D+1}}}{1-\frac{c}{r}}+8\frac{\frac{c^{D+1}}{r^{D+1}}}{1-\frac{c}{r}}\right)\frac{\|\sqrt{U}\|^4}{\mu^2_{\mathrm{min}}}(3\rho\sqrt{s\kappa}\eta\frac{1}{\alpha})\\
&\ \ \ + \left(6c\frac{\frac{1}{r}(1-\frac{1}{r^D})}{1-\frac{1}{r}}+8\frac{\frac{c}{r}(1-\frac{c^D}{r^D})}{1-\frac{c}{r}}\right)\frac{\|\sqrt{U}\|^4}{\mu^2_{\mathrm{min}}}(3\rho\sqrt{s\kappa}\eta\frac{1}{\alpha})\\
&\leq \left(12\frac{c}{r^{D+1}}+16\frac{c^{D+1}}{r^{D+1}}\right)\frac{\|\sqrt{U}\|^4}{\mu^2_{\mathrm{min}}}(3\rho\sqrt{s\kappa}\eta\frac{1}{\alpha})\\
&\ \ \ + \left(6\frac{c D}{r^D}+16\frac{c}{r}\right)\frac{\|\sqrt{U}\|^4}{\mu^2_{\mathrm{min}}}(3\rho\sqrt{s\kappa}\eta\frac{1}{\alpha})\\
&\leq \left(12\frac{c}{r^{D+1}}+16\frac{1}{2^D}\frac{c}{r^{D+1}}\right)\frac{\|\sqrt{U}\|^4}{\mu^2_{\mathrm{min}}}(3\rho\sqrt{s\kappa}\eta\frac{1}{\alpha})\\
&\ \ \ + \left(6\frac{c D}{r^{D+1}}+16\frac{c}{r^{D+1}}\right)\frac{\|\sqrt{U}\|^4}{\mu^2_{\mathrm{min}}}(3\rho\sqrt{s\kappa}\eta\frac{1}{\alpha})
\end{align*}
where we have applied $\frac{1-\frac{1}{r^D}}{1-\frac{1}{r}}=(\frac{1}{r^{D-1}}+...+1)\leq D\frac{1}{r^{D-1}}$ in the first inequality, and $c\leq\frac{1}{2}$ in the second one.

Let us further set $D\geq 2$ so that 
\[
12\frac{c}{r^{D+1}}+16\frac{1}{2^D}\frac{c}{r^{D+1}}+6\frac{c D}{r^{D+1}}+16\frac{c}{r^{D+1}}\leq 25\frac{c D}{r^{D+1}}.
\]
As a result, we obtain
\[
\tilde{c}_1\leq 25\frac{c D}{r^{D+1}}\frac{\|\sqrt{U}\|^4}{\mu^2_{\mathrm{min}}}(3\rho\sqrt{s\kappa}\eta\frac{1}{\alpha}).
\]
Furthermore, \eqref{equ: to exp} tells us that
\[
\frac{1}{r^{D+1}}=\left(1-\frac{1}{\sqrt{s\kappa}}\right)^{-(D+1)}\leq e^{\frac{2(D+1)}{\sqrt{s\kappa}}}\leq e^{\frac{2(D+1)}{\sqrt{\kappa}}}.
\]
So finally, we arrive at
\begin{align}
\tilde{c}_1&\leq 25 c D e^{\frac{2(D+1)}{\sqrt{\kappa}}}\frac{\|\sqrt{U}\|^4}{\mu^2_{\mathrm{min}}}(3\rho\sqrt{s\kappa}\eta\frac{1}{\alpha})=3 b \rho \eta \beta \sqrt{s\kappa}.\label{equ: tilde c1}
\end{align}
where we have set
\begin{align}
\label{equ: b}
b= \frac{25\|\sqrt{U}\|^4}{\alpha\beta\mu^2_{\mathrm{min}}}c D e^{\frac{2(D+1)}{\sqrt{\kappa}}}.
\end{align}

\subsubsection{Determine \texorpdfstring{$\rho$}{} and \texorpdfstring{$s$}{} to make the second and third term of \texorpdfstring{\eqref{equ: E bound}}{} non-positive}
The coefficient $C_3$ in \eqref{equ: E bound} and \eqref{equ: coefficients} satisfies
\begin{align}
    \frac{C_3}{\eta s\kappa}&=-1+2c_1\left(\frac{2}{\sqrt{s\kappa}+1}\right)^2+2\tilde{c}_1\left(\frac{2}{\sqrt{s\kappa}+1}\right)^2\nonumber\\
    &\qquad\qquad+\left(\eta\beta+\frac{\eta\alpha\sqrt{s\kappa}}{\rho}\right)\nonumber\\
    &\leq -1+6(a+b)\rho\eta\beta\sqrt{s\kappa}\left(\frac{2}{\sqrt{s\kappa}+1}\right)^2\nonumber\\
    &\qquad\qquad+\eta\beta\left(1+\frac{\sqrt{s\kappa}}{\rho\kappa}\right)\nonumber\\
    &\leq -1+\eta\beta\left(\frac{24(a+b)\rho}{\sqrt{s\kappa}}+1+\frac{\sqrt{s}}{\rho\sqrt{\kappa}}\right)\nonumber\\
    &\leq -1+\eta\beta\left(\frac{24(a+b)\rho}{\sqrt{s}}+1+\frac{\sqrt{s}}{\rho}\right)\nonumber\\
    &=-1 + 24(a+b)\sqrt{\rho(\rho-1)}(\eta\beta)^{\frac{3}{2}}\nonumber\\
    &\qquad\qquad+\eta\beta+\frac{1}{\sqrt{\rho(\rho-1)}}(\eta\beta)^{\frac{1}{2}},\label{equ: 1}
\end{align}
where in the first step we have applied \eqref{equ: c_1} and \eqref{equ: tilde c1}, and in last step we have used \eqref{equ: s}.

Now, we take $\rho>2$ and
\begin{align}
    \eta&=\frac{\rho-2}{\rho-1}\frac{1}{(1+24(a+b)\rho)\beta}\label{equ: eta},\\
    s &=\frac{\rho}{\rho-1}\frac{1}{\eta\beta}= \frac{\rho(1+24(a+b)\rho)}{\rho-2},\label{equ: ss}
\end{align}
It is evident that $\eta\beta\leq 1$. Consequently, from \eqref{equ: 1} we know that 
\begin{align*}
    &\frac{C_3}{\eta s \kappa}\leq -1+\eta\beta(24(a+b)\rho+1)+\frac{1}{\rho-1}=0.
    \end{align*}
In order to make $s$ defined in \eqref{equ: ss} as small as possible, we minimize the right-hand side of \eqref{equ: ss} with respect to $\rho$ to get
\begin{align}
\label{equ: rho}
    \rho=2+\sqrt{2+\frac{1}{12(a+b)}}>2,
\end{align}
which tells us that 
\begin{align}
&s = \frac{2+\sqrt{2+\frac{1}{12(a+b)}}}{\sqrt{2+\frac{1}{12(a+b)}}}\nonumber\\
&\quad\quad\quad\quad\quad\times\Bigg(1+24(a+b)\left(2+\sqrt{2+\frac{1}{12(a+b)}}\right)\Bigg), \label{equ: sss}\\
&\eta = \frac{\sqrt{2+\frac{1}{12(a+b)}}}{1+\sqrt{2+\frac{1}{12(a+b)}}}\frac{1}{1+24(a+b)(2+\sqrt{2+\frac{1}{12(a+b)}})}\frac{1}{\beta},\label{equ: eta1}
\end{align}
where $a= \frac{6\|\sqrt{U}\|^4}{\alpha\beta\mu^2_{\mathrm{min}}}\gamma D e^{\frac{2D}{\sqrt{\kappa}}}$ and $b=\frac{10\|\sqrt{U}\|^4}{\alpha\beta\mu^2_{\mathrm{min}}}c D e^{\frac{2(D+1)}{\sqrt{\kappa}}}$ are defined in \eqref{equ: a} and \eqref{equ: b}, respectively.

The coefficient $C_2$ in \eqref{equ: E bound} and \eqref{equ: coefficients} satisfies
\begin{align*}
    &C_2=\ \ \ \left(1-\frac{1}{\sqrt{s\kappa}}\right)\left(-\frac{1}{\sqrt{s\kappa}}+\eta\alpha\left(\frac{\sqrt{s\kappa}-1}{\rho}-1\right)\right)\\
    &\quad+2c_1\left(\frac{\sqrt{s\kappa}-1}{\sqrt{s\kappa}(\sqrt{s\kappa}+1)}\right)^2+2\tilde{c}_1\left(\frac{\sqrt{s\kappa}-1}{\sqrt{s\kappa}(\sqrt{s\kappa}+1)}\right)^2\\
    &=\left(1-\frac{1}{\sqrt{s\kappa}}\right)\left(-\frac{1}{\sqrt{s\kappa}}+\eta\alpha\left(\frac{\sqrt{s\kappa}}{\rho}\right)\right)\\
    &+2c_1\left(\frac{\sqrt{s\kappa}-1}{\sqrt{s\kappa}(\sqrt{s\kappa}+1)}\right)^2+2\tilde{c}_1\left(\frac{\sqrt{s\kappa}-1}{\sqrt{s\kappa}(\sqrt{s\kappa}+1)}\right)^2\\
    &= (1-\frac{1}{\sqrt{s\kappa}})\frac{1}{\sqrt{s\kappa}}\left(-1+\frac{\eta s \beta}{\rho}\right) \\
    &\quad + 2(c_1+\tilde{c}_1)\left(\frac{\sqrt{s\kappa}-1}{\sqrt{s\kappa}(\sqrt{s\kappa}+1)}\right)^2\\
    &\leq (1-\frac{1}{\sqrt{s\kappa}})\frac{1}{\sqrt{s\kappa}}\left(-1+\frac{\eta s \beta}{\rho}\right) \\
    &\quad + 6(a+b)\rho\eta\beta \sqrt{s\kappa} \left(\frac{\sqrt{s\kappa}-1}{\sqrt{s\kappa}(\sqrt{s\kappa}+1)}\right)^2\\
    &\leq (1-\frac{1}{\sqrt{s\kappa}})\frac{1}{\sqrt{s\kappa}}\left(-1+\frac{\eta s \beta}{\rho}+6(a+b)\rho\eta\beta\right),
\end{align*}
where in the first inequality, we have applied \eqref{equ: c_1} and \eqref{equ: tilde c1}, and $\frac{1}{\sqrt{s\kappa}}(\frac{\sqrt{s\kappa}-1}{\sqrt{s\kappa}+1})^2< (1-\frac{1}{\sqrt{s\kappa}})\frac{1}{\sqrt{s\kappa}}$ in the second inequality.

Furthermore, we have
\begin{align*}
    &\ \ \ -1+\frac{\eta s\beta}{\rho}+6(a+b)\rho\eta\beta\\
    &=-1 +\frac{1}{\rho-1}+\frac{\rho-2}{\rho-1}\frac{6(a+b)\rho}{(1+24(a+b)\rho)}\\
    &=\frac{\rho-2}{\rho-1}\left(-1+\frac{6(a+b)\rho}{1+24(a+b)\rho}\right)\\
    &\leq 0,
\end{align*}
where we have applied \eqref{equ: s} and \eqref{equ: eta} in the first equality.

Now, we can conclude that the right-hand side of \eqref{equ: E bound} is non-positive. Therefore, $L^k$ converges to $0$ at a linear rate of $1-\frac{1}{\sqrt{s\kappa}}$. Recalling the definition of $L^k$ in \eqref{equ: lyapunov function}, we know that the iteration complexity for Algorithm \ref{alg: global formulation} to obtain an $\varepsilon-$suboptimal solution is  
\begin{align*}
    \mathcal{I}_{\mathrm{DLAG}}=\mathcal{O}\left(\sqrt{s\kappa}\log(\frac{1}{\varepsilon})\right),
\end{align*}
where $s$ is given by \eqref{equ: sss}. 

As a result, if the subproblem \eqref{equ: subproblem} is solved by AGD with $\mathcal{O}(m\sqrt{\kappa_{\mathrm{max}}})$ gradient evaluations, then Algorithm \ref{alg: local formulation} needs $\mathcal{O}\left(m\sqrt{\kappa_{\mathrm{max}}}\sqrt{s\kappa}\log(\frac{1}{\varepsilon})\right)$ gradient evaluations to reach $\vareps-$ suboptimality. If the subproblem \eqref{equ: subproblem} is solved by Katyusha with $\mathcal{O}(m+\sqrt{m\kappa_{\mathrm{max}}})$ stochastic gradient evaluations, then Algorithm \ref{alg: local formulation} needs $\mathcal{O}\left((m+\sqrt{m\kappa_{\mathrm{max}}})\sqrt{s\kappa}\log(\frac{1}{\varepsilon})\right)$ stochastic gradient evaluations to reach $\vareps-$suboptimality in expectation. 

\subsection{Proof of Theorem \ref{thm: iteration complexity}}
\label{app: compare iteration complexity}

By Theorem \ref{thm: iteration complexity} we know that $\mathcal{I}_{\mathrm{DLAG}}(\varepsilon)=\mathcal{O}\left(\sqrt{s\kappa}\log(\frac{1}{\varepsilon})\right)$. In this proof, we will show that $s=10$ under the settings of  Theorem \ref{thm: iteration complexity}, so that $\mathcal{I}_{\mathrm{DLAG}}(\varepsilon)=\mathcal{O}\left(\sqrt{\kappa}\log(\frac{1}{\varepsilon})\right)$.

In fact, from \eqref{equ: a} and \eqref{equ: b} we know that
\[
a = \frac{6\|\sqrt{U}\|^4}{\alpha\beta\mu^2_{\mathrm{min}}}\gamma D e^{\frac{2D}{\sqrt{\kappa}}} = \frac{1}{48},
\]
\[
b = \frac{25\|\sqrt{U}\|^4}{\alpha\beta\mu^2_{\mathrm{min}}}c D e^{\frac{2(D+1)}{\sqrt{\kappa}}} = \frac{1}{48}.
\]
Note that we have $c<\frac{\alpha\beta\mu^2_{\mathrm{min}}}{1200 D\|\sqrt{U}\|^4}<\frac{\alpha\beta\mu^2_{\mathrm{min}}}{\|\sqrt{U}\|^4}\leq\frac{\sigma^2_1(U)}{\|\sqrt{U}\|^4} <1$.

\eqref{equ: sss} tells us that
\begin{align*}
s &= \frac{2+\sqrt{2+\frac{1}{12(a+b)}}}{\sqrt{2+\frac{1}{12(a+b)}}}\\
&\quad\quad\times\left(1+24(a+b)\left(2+\sqrt{2+\frac{1}{12(a+b)}}\right)\right)\\
&= 10.
\end{align*}
And \eqref{equ: eta1} tells us that
\begin{align*}
\eta &= \frac{\sqrt{2+\frac{1}{12(a+b)}}}{1+\sqrt{2+\frac{1}{12(a+b)}}}\frac{1}{1+24(a+b)(2+\sqrt{2+\frac{1}{12(a+b)}})}\frac{1}{\beta}\\
&=\frac{2}{15}\frac{1}{\beta}.
\end{align*}

\subsection{Communication complexity}
\label{app: communication complexity}
In this section, we prove the communication complexity of Algorithm DLAG(\ref{alg: global formulation}) and MDLAG(Algorithm \ref{alg: MDLAG global formulation}) as stated in Theorem \ref{thm: communication complexity}.

To analyze the communication complexity of Algorithm \ref{alg: global formulation}, let us first define the importance factor of worker $i$:
    $$H_i=\frac{1/\mu_i}{1/\mu_{\mathrm{min}}}=\frac{\mu_{\mathrm{min}}}{\mu_i},$$ 
where $\frac{1}{\mu_i}$ is smoothness parameter of $f^*_i$.

We first show that, if $H^2_i\leq\frac{\gamma}{d}$ for some $1\leq d\leq D$, then  worker $i$ communicates with its neighbors at most every $(d+1)$ iterations.

Suppose at iteration $k$, the most recent iteration that worker $i$ sends information to its neighbors is $k-d'$ for some $d'$ that satisfies $1\leq d'\leq d\leq D$, i.e., $\hat{\theta}_i^{k-1}=\theta_i^{k-d'}$. As a result,
\begin{align}
\label{equ: comm1}
\begin{split}
&\|\hat{\theta}^{k-1}_i-\theta^k_i\|^2\\
& \leq  3\|\theta^{k-d'}_i-\nabla f^*_i(x^{k-d'}_i)\|^2+3\|\theta^k_i-\nabla f^*_i(x^k_i)\|^2\\
&\ \ \ \ +3\|\nabla f^*_i(x^{k-d'}_i)-\nabla f^*_i(x^k_i)\|^2.
\end{split}
\end{align}
And \eqref{equ: error1} in the proof of Lemma \ref{lem: error propagating dynamics} tells us that 
\begin{align}
\label{equ: comm2}
&\E\|\theta^k_i-\nabla f^*_i(x^k_i)\|^2\nonumber\\ &\leq  \sum_{j=0}^{k-1}c^{k-j}\E\|\nabla f^*_i(x^{j}_i)-\nabla f^*_i(x^{j+1}_i)\|^2\nonumber\\
&\leq \sum_{j=0}^{k-1}\frac{c^{k-j}}{\mu^2_{\mathrm{min}}}\E\|x^{j}_i- x^{j+1}_i\|^2,
\end{align}
\begin{align}
\begin{split}
\label{equ: comm3}
&\|\theta^{k-d'}_i-\nabla f^*_i(x^{k-d'}_i)\|^2 \\
&\leq  \sum_{j=0}^{k-d'-1}c^{k-d'-j}\E\|\nabla f^*_i(x^{j}_i)-\nabla f^*_i(x^{j+1}_i)\|^2\\
&\leq \sum_{j=0}^{k-d'-1}\frac{c^{k-d'-j}}{\mu^2_{\mathrm{min}}}\E\|x^{j}_i- x^{j+1}_i\|^2\\
&\leq \sum_{j=0}^{k-D-1}\frac{c^{k-D-j}}{\mu^2_{\mathrm{min}}}\E\|x^j_i-x^{j+1}_i\|^2\\ &\quad +\sum_{j=k-D}^{k-1}\frac{c}{\mu^2_{\mathrm{min}}}\E\|x^j_i-x^{j+1}_i\|^2.
\end{split}
\end{align}
Furthermore, we know that
\begin{align}
\label{equ: comm4}
\begin{split}
    &\|\nabla f_i^*(x_i^{k-d'})-\nabla f_i^*(x_i^k)\|^2\\
    &\leq \frac{1}{\mu_i^2}\|x_i^{k-d'}-x_i^k\|^2\leq d'\frac{1}{\mu_{\mathrm{min}}^2} H_i^2\sum\limits_{j=1}^{d'}\|x_i^{k+1-j}-x_i^{k-j}\|^2\\
    &\leq \frac{\gamma}{\mu_{\mathrm{min}}^2}\sum\limits_{j=1}^{d'}\|x_i^{k+1-j}-x_i^{k-j}\|^2\leq \frac{\gamma}{\mu_{\mathrm{min}}^2}\sum\limits_{j=1}^{D}\|x_i^{k+1-j}-x_i^{k-j}\|^2.
\end{split}
\end{align}
Applying \eqref{equ: comm2}, \eqref{equ: comm3}, and \eqref{equ: comm4} to \eqref{equ: comm1}, we arrive at
\begin{align*}
\begin{split}
\E\|\hat{\theta}^{k-1}_i-\theta^k_i\|^2  & \leq  3\sum_{j=0}^{k-D-1}\frac{c^{k-D-j}}{\mu^2_{\mathrm{min}}}\E\|x^j_i-x^{j+1}_i\|^2\\
&\quad +3\sum_{j=0}^{k-1}\frac{c^{k-j}}{\mu^2_{\mathrm{min}}}\E\|x^j_i-x^{j+1}_i\|^2\\
&\quad + 3\sum_{j=k-D}^{k-1}\frac{c}{\mu^2_{\mathrm{min}}}\E\|x^j_i-x^{j+1}_i\|^2\\
&\quad +3\sum_{j=k-D}^{k-1}\frac{\gamma}{\mu^2_{\mathrm{min}}}\E\|x^j_i-x^{j+1}_i\|^2.
\end{split}
\end{align*}
As a result, worker $i$'s lazy condition at iteration $k$ is satisfied in expectation. Since $d'$ can be any integer in $[1,d]$, we know that worker $i$ send gradients to its neighbors at most every $(d+1)$ iterations in expectation.

To obtain the communication complexity of Algorithm \ref{alg: global formulation}, we recall the definition of the heterogeneity score function $h_d(\gamma)$ for $d=1,2,...,D$:
\begin{align}
\label{equ: score function 1}
    h_d(\gamma)=\frac{1}{2|\mathcal{E}|}\sum\limits_{i=1}^n m_i\mathbb{1}\left(H_i^2\leq\frac{\gamma}{d}\right), 
\end{align}
where $|\mathcal{E}|$ is the number of edges in the network, $m_i$ is the number of edges connected to worker $i$. We have $\sum_{i=1}^n m_i=2|\mathcal{E}|$. $\mathbb{1}\left(H_i^2\leq\frac{\gamma}{d}\right)$ equals $1$ if $H_i^2\leq\frac{\gamma}{d}$, and $0$ otherwise.

Now, let us split all $n$ workers into $(D+1)$ subgroups:

$\mathcal{M}_0:$ every worker $i$ that does not satisfy $H_i^2\leq \gamma$;

...

$\mathcal{M}_d:$ every worker $i$ that does not satisfy $H_i^2\leq \frac{\gamma}{d+1}$ but satisfies $H_i^2\leq \frac{\gamma}{d}$;

...

$\mathcal{M}_D:$ every worker $i$ that satisfies $H_i^2\leq \frac{\gamma}{D}$;

Then the communication complexity for DLAG(Algorithm \ref{alg: global formulation}) to reach $\varepsilon-$suboptimality satisfies
\begin{align*}
    \mathcal{C}_{\mathrm{DLAG}}(\varepsilon)
    &\leq \sum_{d=0}^D \sum_{i\in\mathcal{M}_d}m_i\frac{\mathcal{I}_{\mathrm{DLAG}}(\varepsilon)}{d+1}\\
    &\overset{(a)}{\leq}\bigg(1-h_1(\gamma)+\frac{1}{2}(h_1(\gamma)-h_2(\gamma))\\
    &\quad\quad\quad\quad+\cdots+\frac{1}{D+1}h_D(\gamma)\bigg)2|\mathcal{E}|\mathcal{I}_{\mathrm{DLAG}}(\varepsilon)\\
    &=\left(1-\sum\limits_{d=1}^D\left(\frac{1}{d}-\frac{1}{d+1}\right)h_d(\gamma)\right)2|\mathcal{E}|\mathcal{I}_{\mathrm{DLAG}}(\varepsilon),
\end{align*}
where (a) follows from \eqref{equ: score function 1} and the definition of the subgroups $\mathcal{M}_d$ for $d=0,2,...,D$.

For MDLAG(Algorithm \ref{alg: MDLAG global formulation} or \ref{alg: global formulation}) with $K$ communication rounds per iteration, The communication save happens at the first round of communication. Therefore, we have
\begin{align*}
    &\mathcal{C}_{\mathrm{MDLAG}}(\varepsilon)\\
    &\leq (K-1)2|\mathcal{E}|\mathcal{I}_{MDLAG}(\vareps)+\sum_{d=0}^D \sum_{i\in\mathcal{M}_d}m_i\frac{\mathcal{I}_{\mathrm{MDLAG}}(\varepsilon)}{d+1}\\
    &\leq(K-1)2|\mathcal{E}|\mathcal{I}_{MDLAG}(\vareps)+\bigg(1-h_1(\gamma)+\frac{1}{2}(h_1(\gamma)\\
    &\quad\quad\quad\quad-h_2(\gamma))+\cdots+\frac{1}{D+1}h_D(\gamma)\bigg)2|\mathcal{E}|\mathcal{I}_{\mathrm{MDLAG}}(\varepsilon)\\
    &=\left(K-\sum\limits_{d=1}^D\left(\frac{1}{d}-\frac{1}{d+1}\right)h_d(\gamma)\right)2|\mathcal{E}|\mathcal{I}_{\mathrm{MDLAG}}(\varepsilon),
\end{align*}

\subsubsection{An illustrative example for Corollary \ref{coro: compare communication complexity}}
\begin{example}
\label{exam: example}
If $\mu_1=\mu_{\mathrm{min}}\leq \sqrt{\frac{\gamma}{D}}$, and  $\mu_2=\mu_3=...=\mu_{n}= 1$, then $h_d\left(\gamma\right)\equiv 1-\frac{m_1}{2|\mathcal{E}|}$\,\, for $d=1,2,...,D$. Furthermore,
\begin{align*}
\frac{\mathcal{C}_{\mathrm{DLAG}}(\varepsilon)}{\mathcal{C}_{SSDA}(\varepsilon)}&\leq \sqrt{10}\left(1-\left(1-\frac{1}{D+1}\right)\left(1-\frac{m_1}{2|\mathcal{E}|}\right)\right) \\
&\leq \sqrt{10}\left(\frac{1}{D+1}+\frac{m_1}{2|\mathcal{E}|}\right).
\end{align*}
As a result, $\frac{\mathcal{C}_{\mathrm{DLAG}}(\varepsilon)}{\mathcal{C}_{SSDA}(\varepsilon)}< \frac{1}{3}$ when $D=20$ and $\frac{m_1}{2|\mathcal{E}|}\leq\frac{1}{20}$.
\end{example}

\subsection{MDLAG and Proof of Theorem \ref{thm: MDLAG iteration complexity}}
\label{app: MDLAG}
In this section, we first provide a full description of the MDLAG algorithm as in Algorithm \ref{alg: MDLAG global formulation}. Then, we prove its iteration complexity and stochastic gradient complexity stated in Theorem \ref{thm: MDLAG iteration complexity}.
\begin{algorithm}[H]
\caption{Multi-DLAG (MDLAG)}
\label{alg: MDLAG local formulation}
    \textbf{Input:} $x^0_i=y^0_i=0$ and $\hat{\theta}^0_i=\theta^0_i=\nabla f^*_i(x^0_i)$ for workers $i=1,2,...,n$, step size $\eta>0$, parameters $s\geq 1$, $K=\lfloor\frac{1}{\sqrt{\zeta(U)}}\rfloor, c_1 = \frac{1-\sqrt{\zeta(U)}}{1+\sqrt{\zeta(U)}}, c_2 = \frac{1+\zeta(U)}{1-\zeta(U)}, c_3=\frac{2}{(1+\zeta(U))\sigma_1(U)}, \kappa'=\frac{\kappa_F}{\zeta(P_K(U))}$.
    \begin{flushleft}
    \textbf{Output:} $y^K=(y^K_1, y^K_2, ...,y^K_n).$
    \end{flushleft}
    \begin{algorithmic}[1]
        \FOR{each worker $i$ in parallel}
        \STATE{Read $P_i^{k-1}$, $\hat{\theta}^{k-1}_i$, and $\theta^{k-1}_i$ from cache;} 
        \STATE{Get $\theta^k_i$ via {\small $\mathcal{O}\left((m+\sqrt{m\kappa_{\mathrm{max}}})\log(\frac{2\kappa_{\mathrm{max}}}{c})\right)$} stochastic gradient steps of Katyusha, warm started at $\theta^{k-1}_i$;}
        \IF{$\hat{\theta}^{k-1}_i$
        fails condition \eqref{equ: worker's condition 1} \textbf{or} $d^{k-1}_i=D$}
        \STATE{Send $Q_i^k\coloneqq\theta^k_i-\hat{\theta}^{k-1}_{i}$ to worker $i'\in\mathcal{N}(i)\backslash \{i\}$;}
        \STATE{$\hat{\theta}^{k}_{i}\leftarrow {\theta}^{k}_{i}$;}
        \STATE{$d^{k}_i=0$;}
        \ELSE{}
        \STATE{(Worker $i$ sends out nothing)}
        \STATE{$\hat{\theta}^{k}_i\leftarrow \hat{\theta}^{k-1}_i$;}
        \STATE{$d^k_i=d^{k-1}_i+1$;}
        \ENDIF
        \ENDFOR
        \STATE{$y^{k+1}\leftarrow x^k-\eta\,\, \text{Accelerated Gossip}(\hat{\Theta}^k, U, K)$;}
        \STATE{$x^{k+1}\leftarrow y^{k+1}+\frac{\sqrt{s\kappa'}-1}{\sqrt{s\kappa'}+1}(y^{k+1}-y^k)$;}
        \STATE{$k\leftarrow k+1$;}
        \STATE{\textbf{Procedure} Accelerated Gossip$(z, U, K)$}
        \STATE{$a_0=1$, $a_1=c_2;$}
        \FOR{each worker $i$ in parallel}
        \STATE{Let $S^k_i\coloneqq\{j\in\mathcal{N}(i)\mid j \,\,\text{sends out}\,\, Q^k_j\}$;}
        \STATE{Update cache: $P_i^k\leftarrow P_i^{k-1}+\sum_{j\in S^k_i }U_{ij}Q_j^k$;}
        \ENDFOR
        \STATE{$z_0=z, z_1 = c_2 z - c_2c_3 P^k;$}
        \FOR{$l = 1$ to $K-1$}
        \STATE{$a_{l+1}= 2c_2 a_l - a_{l-1}$;}
        \STATE{$z_{l+1}=2c_2z_l(I-c_3U)-z_{l-1};$}
        \ENDFOR
        \STATE{\textbf{return} $z_0-\frac{z_K}{a_K};$}
        \STATE{\textbf{end Procedure}}
    \end{algorithmic}
\end{algorithm}
Compared with DLAG (Algorithm \ref{alg: local formulation}), MDLAG applies an accelerated gossip procedure in line 14, where $K$ rounds of communications are performed instead of $1$ round, and communication save happens at the first round. This procedure is summarized in lines 17-29.

\subsubsection{Proof of Theorem \ref{thm: MDLAG iteration complexity}}
Compared with DLAG, MDLAG applies $P_K(U)$ as the gossip matrix instead of $U$, where $P_K(U)=I-\frac{T_K(c_2(I-U))}{T_K(c_2I)},$ $K=\lfloor\frac{1}{\sqrt{\zeta(U)}}\rfloor$,  $c_2=\frac{1+\gamma}{1-\gamma}$, and $T_K$ is a Chebyshev polynomial of power $K$. By Theorem 4 of \cite{scaman2017optimal}, $P_K(U)$ is also a gossip matrix satisfying Assumption \ref{assump: assumption on network topology}, and its eigengap satisfies $\zeta(P_K(U))\geq \frac{1}{4}.$

In this regard, MDLAG can be viewed as Nesterov's accelerated gradient descent applied to the following problem\footnotemark[1], but with inexact dual gradients given by Kaytusha with warm start, and the worker's lazy condition \eqref{equ: worker's condition 1}:
\begin{align}
\label{equ: MDLAG dual problem}
    \mathop{\mathrm{minimize}}_{\xi\in\RR^{d\times n}} G(\xi)\coloneqq F^*(\xi \sqrt{P_K(U)}).
\end{align}

\footnotetext[1]{Recall that DLAG is Nesterov's accelerated gradient descent applied to problem \eqref{equ: dual problem}, but with inexact dual gradients given by Kaytusha with warm start, and the worker's lazy condition \eqref{equ: worker's condition 1}.}

Therefore, the iteration and stochastic gradient complexity of MDLAG can be derived in a similar fashion as those of DLAG, the only difference is that the gossip matrix $U$ is replaced by $P_K(U)$. 

Similar to DLAG, let us apply the following parameter choices: $\gamma=\frac{\alpha'\beta'\mu^2_{\mathrm{min}}}{288D\|\sqrt{P_K(U)}\|^4}e^{-\frac{2D}{\sqrt{\kappa_F}}}$, $c=\frac{\alpha'\beta'\mu^2_{\mathrm{min}}}{1200 D\|\sqrt{P_K(U)}\|^4}e^{-\frac{2(D+1)}{\sqrt{\kappa_F}}}<1$, $\eta=\frac{2}{15}\frac{1}{\beta'}$ and $s = 10$, where $\alpha' = \frac{\sigma_{n-1}(P_K(U))}{L_{\max}}$ and $\beta'=\frac{\sigma_1(P_K(U))}{\mu_{\min}}$. Then, the iteration and stochastic gradient complexities of MDLAG are given by
\begin{align*}
\mathcal{I}_{\mathrm{MDLAG}}(\varepsilon)=\mathcal{O}\left(\sqrt{\kappa'}\log(\frac{1}{\varepsilon})\right)
\end{align*}
\begin{align*}
\mathcal{G}_{\mathrm{MDLAG}}(\varepsilon)=\mathcal{O}\left(n(m+\sqrt{m\kappa_{\mathrm{max}}})\sqrt{\kappa'}\log(\frac{1}{\varepsilon})\right),
\end{align*}
where $\kappa'=\frac{\kappa_F}{\zeta(P_K(U))}$ is the condition number of problem \eqref{equ: MDLAG dual problem}, Finally, we notice that $\zeta(P_K(U))\geq \frac{1}{4}$ gives $\kappa'\leq 4 \kappa_F$, which concludes the proof.

\subsection{Comparison with ADFS \cite{hendrikx2019accelerated}}
\label{app: compare with ADFS}

In this section, we compare the performance of DLAG, MDLAG, and ADFS \cite{hendrikx2019accelerated} on cross-entropy minimization. 

To ensure that a fair comparison, We test on the \texttt{covtype} dataset and apply the recommended parameter settings in \cite{hendrikx2019accelerated} for ADFS. Specifically,
\begin{enumerate}
    \item The decentralized network is 10x10 2D grid.
    \item All the other settings are the same as before, except for our DLAG and MDLAG, we set $\gamma=c=1e-5$, and apply 300 epochs of accelerated gradient descent to obtain an approximate dual gradient.
\end{enumerate}
For SSDA and MSDA, we found that it is very prohibitive to apply accelerated gradient descent to solve their subproblems until reaching a high accuracy such as $1e-10$, which would lead to much larger computation complexities. Therefore, their results are not included.
 \begin{figure}[H]
 \centering
 \includegraphics[width=0.38\textwidth]{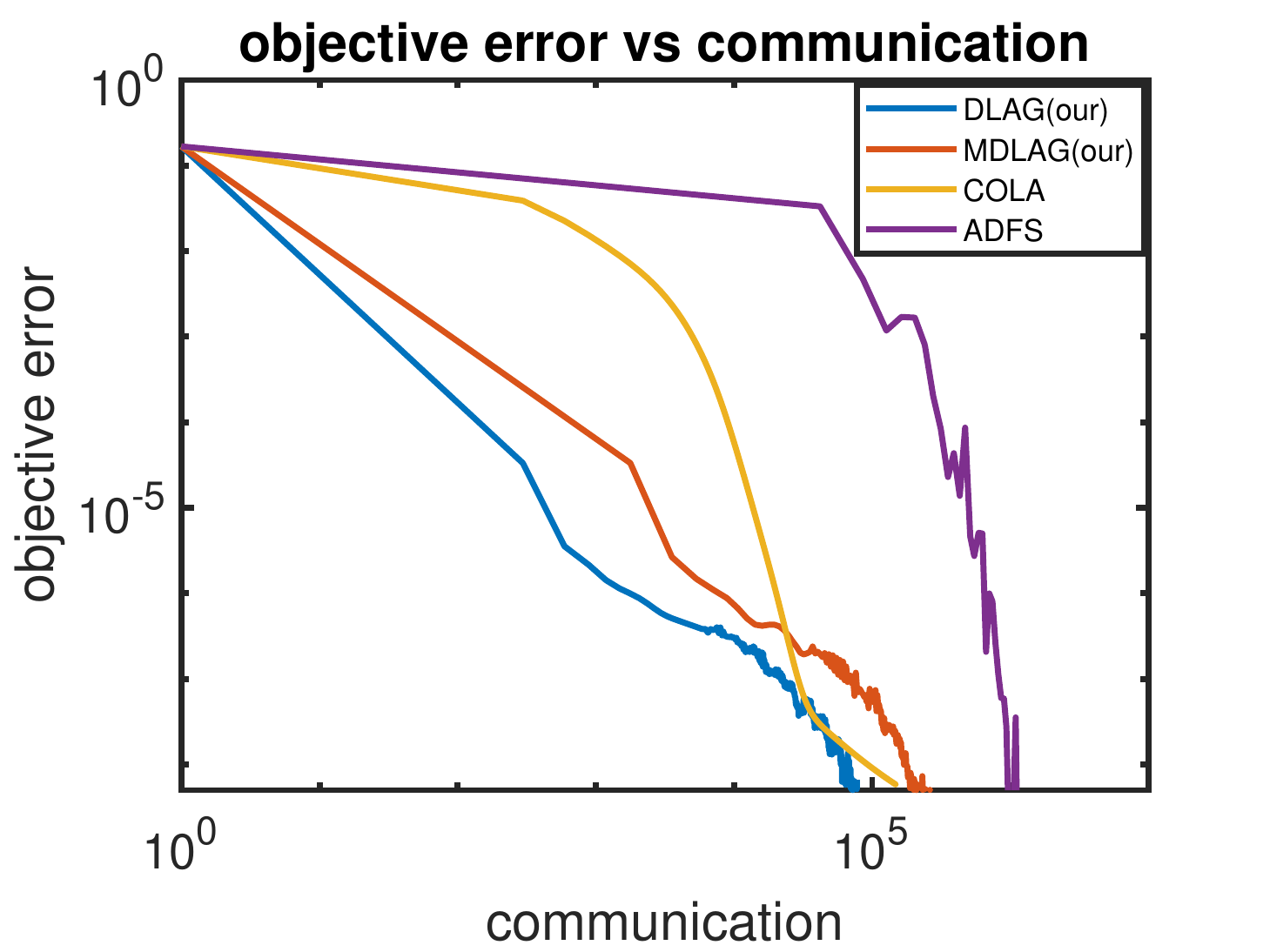}
 \caption{\footnotesize Communication complexities on \texttt{covtype} dataset.}
 \label{fig: covtype_comm}
 \end{figure}
 \begin{figure}[H]
 \centering
 \includegraphics[width=0.38\textwidth]{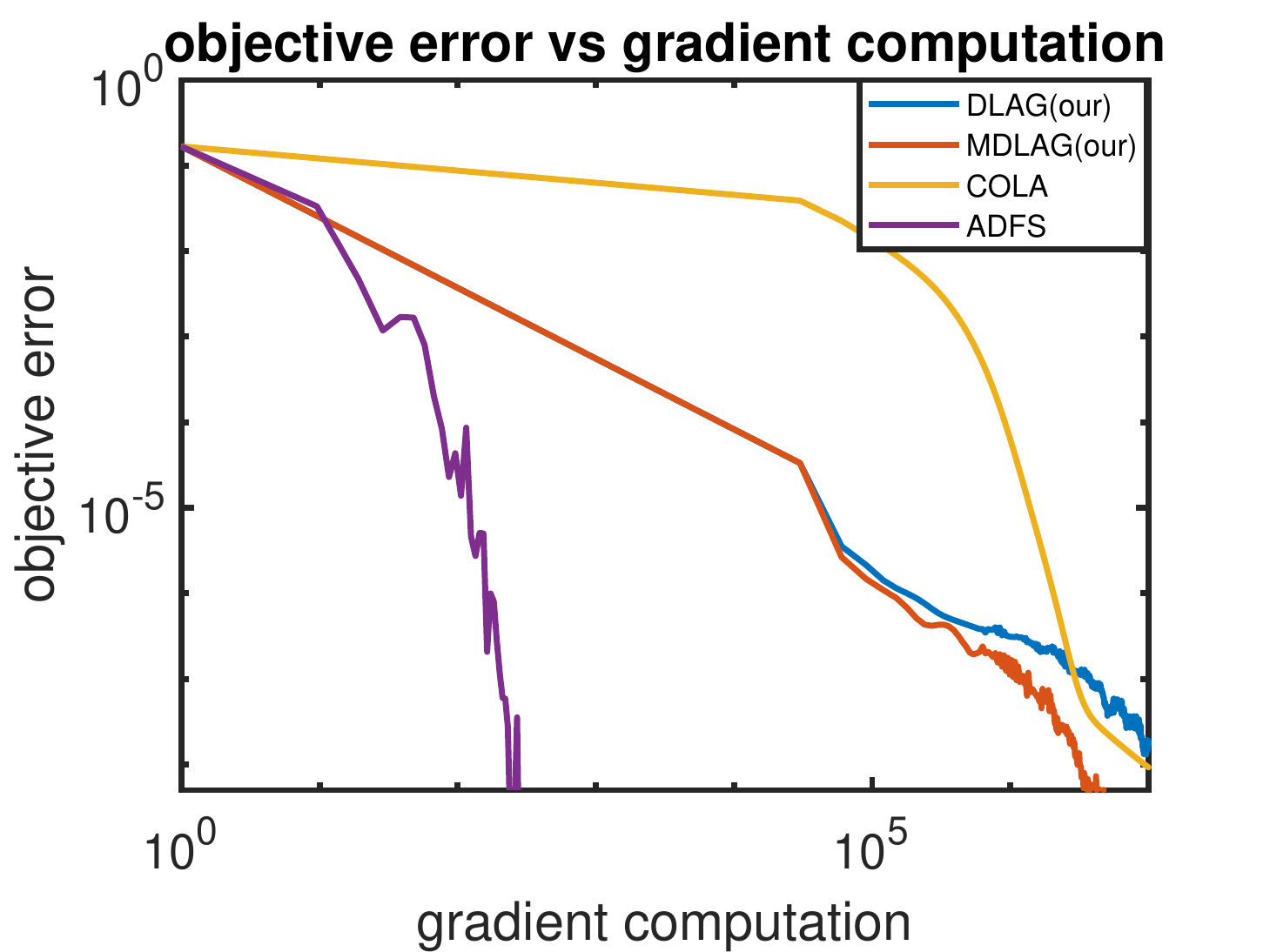}
 \caption{\footnotesize Stochastic gradient complexities on \texttt{covtype} dataset.}
 \label{fig: covtype_gradient}
 \end{figure}
From Figures \ref{fig: covtype_comm} and \ref{fig: covtype_gradient} we can see that ADFS needs more communication than DLAG and MDLAG, as it is not designed to optimize communication complexity. However, ADFS has a better stochastic gradient complexity. This is because at each iteration, ADFS solves a 1-D subproblem that is much simpler than the subproblem \eqref{equ: subproblem} of DLAG and MDLAG, this 1-D subproblem is solved approximately by 10 steps of Newton iterations with warm start.

Finally, we would like to emphasize that our algorithms DLAG and MDLAG aim at making SSDA and MSDA practical for problems without cheap dual gradients, and to reduce their communication complexity, while ADFS focuses on optimizing the running time. It is interesting to ask whether our theory for inexact dual gradients can be generalized to provide convergence guarantee for the inexactly solved subproblems in ADFS, where warm start and fixed number of Newton steps are applied.

\bibliographystyle{IEEEtran}
\bibliography{IEEEabrv, references}

\begin{thebibliography}{10}
\providecommand{\url}[1]{#1}
\csname url@samestyle\endcsname
\providecommand{\newblock}{\relax}
\providecommand{\bibinfo}[2]{#2}
\providecommand{\BIBentrySTDinterwordspacing}{\spaceskip=0pt\relax}
\providecommand{\BIBentryALTinterwordstretchfactor}{4}
\providecommand{\BIBentryALTinterwordspacing}{\spaceskip=\fontdimen2\font plus
\BIBentryALTinterwordstretchfactor\fontdimen3\font minus
  \fontdimen4\font\relax}
\providecommand{\BIBforeignlanguage}[2]{{%
\expandafter\ifx\csname l@#1\endcsname\relax
\typeout{** WARNING: IEEEtran.bst: No hyphenation pattern has been}%
\typeout{** loaded for the language `#1'. Using the pattern for}%
\typeout{** the default language instead.}%
\else
\language=\csname l@#1\endcsname
\fi
#2}}
\providecommand{\BIBdecl}{\relax}
\BIBdecl

\bibitem{scaman2017optimal}
K.~Scaman, F.~Bach, S.~Bubeck, Y.~T. Lee, and L.~Massouli{\'e}, ``Optimal
  algorithms for smooth and strongly convex distributed optimization in
  networks,'' in \emph{Proceedings of the 34th International Conference on
  Machine Learning-Volume 70}.\hskip 1em plus 0.5em minus 0.4em\relax JMLR.
  org, 2017, pp. 3027--3036.

\bibitem{shi2015extra}
W.~Shi, Q.~Ling, G.~Wu, and W.~Yin, ``Extra: An exact first-order algorithm for
  decentralized consensus optimization,'' \emph{SIAM Journal on Optimization},
  vol.~25, no.~2, pp. 944--966, 2015.

\bibitem{nesterov2013introductory}
Y.~Nesterov, \emph{{Introductory Lectures on Convex Optimization: A Basic
  Course}}.\hskip 1em plus 0.5em minus 0.4em\relax Springer Science \& Business
  Media, 2013, vol.~87.

\bibitem{allen2017katyusha}
Z.~Allen-Zhu, ``Katyusha: The first direct acceleration of stochastic gradient
  methods,'' \emph{The Journal of Machine Learning Research}, vol.~18, no.~1,
  pp. 8194--8244, 2017.

\bibitem{chen2018lag}
T.~Chen, G.~Giannakis, T.~Sun, and W.~Yin, ``Lag: Lazily aggregated gradient
  for communication-efficient distributed learning,'' in \emph{Advances in
  Neural Information Processing Systems}, 2018, pp. 5050--5060.

\bibitem{mota2013d}
J.~F. Mota, J.~M. Xavier, P.~M. Aguiar, and M.~P{\"u}schel, ``D-admm: A
  communication-efficient distributed algorithm for separable optimization,''
  \emph{IEEE Transactions on Signal Processing}, vol.~61, no.~10, pp.
  2718--2723, 2013.

\bibitem{shi2014linear}
W.~Shi, Q.~Ling, K.~Yuan, G.~Wu, and W.~Yin, ``On the linear convergence of the
  admm in decentralized consensus optimization,'' \emph{IEEE Transactions on
  Signal Processing}, vol.~62, no.~7, pp. 1750--1761, 2014.

\bibitem{yuan2018exact}
K.~Yuan, B.~Ying, X.~Zhao, and A.~H. Sayed, ``Exact diffusion for distributed
  optimization and learning—part i: Algorithm development,'' \emph{IEEE
  Transactions on Signal Processing}, vol.~67, no.~3, pp. 708--723, 2018.

\bibitem{nedic2017achieving}
A.~Nedic, A.~Olshevsky, and W.~Shi, ``Achieving geometric convergence for
  distributed optimization over time-varying graphs,'' \emph{SIAM Journal on
  Optimization}, vol.~27, no.~4, pp. 2597--2633, 2017.

\bibitem{he2018cola}
L.~He, A.~Bian, and M.~Jaggi, ``Cola: Decentralized linear learning,'' in
  \emph{Advances in Neural Information Processing Systems}, 2018, pp.
  4536--4546.

\bibitem{uribe2018dual}
C.~A. Uribe, S.~Lee, A.~Gasnikov, and A.~Nedi{\'c}, ``A dual approach for
  optimal algorithms in distributed optimization over networks,'' \emph{arXiv
  preprint arXiv:1809.00710}, 2018.

\bibitem{sun2019distributed}
H.~Sun and M.~Hong, ``Distributed non-convex first-order optimization and
  information processing: Lower complexity bounds and rate optimal
  algorithms,'' \emph{IEEE Transactions on Signal processing}, vol.~67, no.~22,
  pp. 5912--5928, 2019.

\bibitem{hendrikx2019accelerated}
H.~Hendrikx, F.~Bach, and L.~Massoulie, ``An accelerated decentralized
  stochastic proximal algorithm for finite sums,'' \emph{arXiv preprint
  arXiv:1905.11394}, 2019.

\bibitem{koloskova2019decentralized}
A.~Koloskova, S.~Stich, and M.~Jaggi, ``Decentralized stochastic optimization
  and gossip algorithms with compressed communication,'' in \emph{International
  Conference on Machine Learning}, 2019, pp. 3478--3487.

\bibitem{mokhtari2016dsa}
A.~Mokhtari and A.~Ribeiro, ``Dsa: Decentralized double stochastic averaging
  gradient algorithm,'' \emph{The Journal of Machine Learning Research},
  vol.~17, no.~1, pp. 2165--2199, 2016.

\bibitem{shen2018towards}
Z.~Shen, A.~Mokhtari, T.~Zhou, P.~Zhao, and H.~Qian, ``Towards more efficient
  stochastic decentralized learning: Faster convergence and sparse
  communication,'' in \emph{International Conference on Machine Learning},
  2018, pp. 4631--4640.

\bibitem{sun2019improving}
H.~Sun, S.~Lu, and M.~Hong, ``Improving the sample and communication complexity
  for decentralized non-convex optimization: A joint gradient estimation and
  tracking approach,'' \emph{arXiv preprint arXiv:1910.05857}, 2019.

\bibitem{seide20141}
F.~Seide, H.~Fu, J.~Droppo, G.~Li, and D.~Yu, ``1-bit stochastic gradient
  descent and its application to data-parallel distributed training of speech
  dnns,'' in \emph{Fifteenth Annual Conference of the International Speech
  Communication Association}, 2014.

\bibitem{strom2015scalable}
N.~Strom, ``Scalable distributed dnn training using commodity gpu cloud
  computing,'' in \emph{Sixteenth Annual Conference of the International Speech
  Communication Association}, 2015.

\bibitem{alistarh2017qsgd}
D.~Alistarh, D.~Grubic, J.~Li, R.~Tomioka, and M.~Vojnovic, ``Qsgd:
  Communication-efficient sgd via gradient quantization and encoding,'' in
  \emph{Advances in Neural Information Processing Systems}, 2017, pp.
  1709--1720.

\bibitem{bernstein2018signsgd}
J.~Bernstein, Y.-X. Wang, K.~Azizzadenesheli, and A.~Anandkumar, ``Signsgd:
  Compressed optimisation for non-convex problems,'' in \emph{International
  Conference on Machine Learning}, 2018, pp. 559--568.

\bibitem{stich2018sparsified}
S.~U. Stich, J.-B. Cordonnier, and M.~Jaggi, ``Sparsified sgd with memory,'' in
  \emph{Advances in Neural Information Processing Systems}, 2018, pp.
  4447--4458.

\bibitem{alistarh2018convergence}
D.~Alistarh, T.~Hoefler, M.~Johansson, N.~Konstantinov, S.~Khirirat, and
  C.~Renggli, ``The convergence of sparsified gradient methods,'' in
  \emph{Advances in Neural Information Processing Systems}, 2018, pp.
  5973--5983.

\bibitem{wangni2018gradient}
J.~Wangni, J.~Wang, J.~Liu, and T.~Zhang, ``Gradient sparsification for
  communication-efficient distributed optimization,'' in \emph{Advances in
  Neural Information Processing Systems}, 2018, pp. 1299--1309.

\bibitem{lin2018don}
T.~Lin, S.~U. Stich, K.~K. Patel, and M.~Jaggi, ``Don't use large mini-batches,
  use local sgd,'' \emph{arXiv preprint arXiv:1808.07217}, 2018.

\bibitem{stich2019local}
S.~U. Stich, ``Local sgd converges fast and communicates little,'' in
  \emph{ICLR 2019 ICLR 2019 International Conference on Learning
  Representations}, no. CONF, 2019.

\bibitem{yu2018parallel}
H.~Yu, S.~Yang, and S.~Zhu, ``Parallel restarted sgd for non-convex
  optimization with faster convergence and less communication,'' \emph{arXiv
  preprint arXiv:1807.06629}, 2018.

\bibitem{wang2018cooperative}
J.~Wang and G.~Joshi, ``Cooperative sgd: A unified framework for the design and
  analysis of communication-efficient sgd algorithms,'' \emph{arXiv preprint
  arXiv:1808.07576}, 2018.

\bibitem{sayed2014adaptation}
A.~H. Sayed \emph{et~al.}, ``Adaptation, learning, and optimization over
  networks,'' \emph{Foundations and Trends{\textregistered} in Machine
  Learning}, vol.~7, no. 4-5, pp. 311--801, 2014.

\bibitem{johnson2013accelerating}
R.~Johnson and T.~Zhang, ``Accelerating stochastic gradient descent using
  predictive variance reduction,'' in \emph{Advances in neural information
  processing systems}, 2013, pp. 315--323.

\bibitem{hendrikx2018accelerated}
H.~Hendrikx, L.~Massouli{\'e}, and F.~Bach, ``Accelerated decentralized
  optimization with local updates for smooth and strongly convex objectives,''
  \emph{arXiv preprint arXiv:1810.02660}, 2018.

\bibitem{young2014iterative}
D.~M. Young, \emph{Iterative solution of large linear systems}.\hskip 1em plus
  0.5em minus 0.4em\relax Elsevier, 2014.

\bibitem{wilson2016lyapunov}
A.~C. Wilson, B.~Recht, and M.~I. Jordan, ``A lyapunov analysis of momentum
  methods in optimization,'' \emph{arXiv preprint arXiv:1611.02635}, 2016.

\end{thebibliography}


\end{document}